\let\ORIlabel\label
\let\ORIrefstepcounter\refstepcounter
\AddToHook{package/hyperref/before}{
   \let\label\ORIlabel
   \let\refstepcounter\ORIrefstepcounter
}
\documentclass[final,onefignum,onetabnum]{siamart171218}
\usepackage{comment}
\usepackage{graphicx}
\usepackage{booktabs}
\usepackage{tikz}
\usepackage{pgfplots}
\usepackage{subcaption}
\usepackage{float}
\pgfplotsset{compat=1.18}
\usetikzlibrary{arrows.meta,bending,decorations.markings}

\usepackage{amsfonts}
\usepackage{amssymb}
\usepackage{bm}
\usepackage{graphicx}
\usepackage{epstopdf}
\usepackage{algorithmic}

\allowdisplaybreaks

\ifpdf
  \DeclareGraphicsExtensions{.eps,.pdf,.png,.jpg}
\else
  \DeclareGraphicsExtensions{.eps}
\fi

\newsiamthm{prop}{Proposition}

\newsiamremark{remark}{Remark}
\newsiamremark{hypothesis}{Hypothesis}

\crefname{hypothesis}{Hypothesis}{Hypotheses}

\newsiamthm{claim}{Claim}

\headers
{FFT-accelerated BIE for smooth surfaces}
{Wenmao Hua, Jun Lai, Huiyi Li and Wangtao Lu}

\title{An FFT-accelerated boundary integral equation method for wave scattering by smooth surfaces in three dimensions}

\author{
Wenmao Hua\thanks{Academy of Mathematics and Systems Science, Chinese Academy of Sciences, Beijing 100190, China (\email{huawenmao@amss.ac.cn}).}
\and
Jun Lai\thanks{School of Mathematical Sciences, Zhejiang University, Hangzhou, Zhejiang 310058, China (\email{laijun6@zju.edu.cn}).}
\and
Huiyi Li\thanks{School of Mathematical Sciences, Zhejiang University, Hangzhou, Zhejiang 310058, China (\email{lililihuiy@zju.edu.cn}).}
\and
Wangtao Lu\thanks{School of Mathematical Sciences, Zhejiang University, Hangzhou, Zhejiang 310058, China (\email{wangtaolu@zju.edu.cn}).}
}

\usepackage{amsopn}

\def\Xint#1{\mathchoice
{\XXint\displaystyle\textstyle{#1}}%
{\XXint\textstyle\scriptstyle{#1}}%
{\XXint\scriptstyle\scriptscriptstyle{#1}}%
{\XXint\scriptscriptstyle\scriptscriptstyle{#1}}%
\!\int}

\def\XXint#1#2#3{{\setbox0=\hbox{$#1{#2#3}{\int}$}
\vcenter{\hbox{$#2#3$}}\kern-.5\wd0}}

\def\cint{\Xint-}
\def\xint{\Xint\times}

\newcommand{\bi}{\mathbf{i}}

\newcommand{\br}{\bm r}
\newcommand{\brs}{\bm r_s}
\newcommand{\bnu}{\bm \nu}

\newcommand{\dS}{d\sigma}

\newcommand{\tphi}{\tilde{\phi}}
\newcommand{\tpsi}{\tilde{\psi}}

\newcommand{\Fr}[1]{\langle #1 \rangle}

\makeatletter
\newcommand*{\addFileDependency}[1]{
  \typeout{(#1)}
  \@addtofilelist{#1}
  \IfFileExists{#1}{}{\typeout{No file #1.}}
}
\makeatother

\ifpdf
\hypersetup{
  pdftitle={An Example Article},
  pdfauthor={Jun Lai, Huiyi Li, and Wangtao Lu}
}
\fi

\newcommand{\ii}{\mathbf{i}}
\newcommand{\FC}[1]{\left\langle #1\right\rangle}

\begin{document}

\maketitle

\begin{abstract}
For wave scattering by axisymmetric surfaces, the fast Fourier transform (FFT)
method provides an effective tool to accelerate standard boundary integral
equation (BIE) solvers.  Surface BIEs can be decoupled into a
series of curve integral equations on the generating curve, due to the
convolution-like integral operators.  The Fourier coefficients of the
three-dimensional fundamental kernels can be rapidly computed through three-term
recurrence relations based on Miller's algorithm. Such well-established
techniques break down for nonaxisymmetric surfaces.

This paper proposes a novel FFT-accelerated boundary integral method for wave scattering by smooth surfaces of arbitrary shapes. The Fourier coefficients of the singular kernels now satisfy higher-order recurrence relations. Although they can be solved with an optimal linear complexity by the standard Olver's algorithm, it turns out that a singularity swapping approach that rewrites each kernel as the product of a smooth function and an axisymmetric-related singular factor is realistically much faster. Consequently, Miller's algorithm together with the standard FFT convolution yields an ${\cal O}(M\log M)$ approach for evaluating the ${\cal O}(M)$ Fourier coefficients of the kernels, attaining exactly the same order of complexity for axisymmetric surfaces! With such FFT-based efficient procedures, we rewrite the surface BIEs in terms of ${\cal O}(M)$ curve integrals, which are proved to  exhibit logarithmic singularities, discretize them by panel-based generalized Gaussian quadratures, and obtain spectrally accurate linear systems to approximate the wavefields. Extensive numerical experiments are carried out to demonstrate the effectiveness and spectral accuracy of the new approach.

\end{abstract}



\section{Introduction}

Wave scattering problems arise in a broad range of applications, including optics,
radar, remote sensing, and seismology. Among the
various existing numerical solvers, boundary integral equation
(BIE) methods are attractive and quite promising, since they can directly handle
radiation conditions at infinity without truncating the unbounded domains and
establish governing integral equations on boundaries of scatterers, thus reducing the
overall dimensionality by one
\cite{colkre13}.

In this paper, we develop a high-accuracy BIE method for three-dimensional (3D)
acoustic scattering by smooth surfaces.
In the existing literature, various high-order discretization techniques have
been developed for two-dimensional (2D) BIEs on smooth
\cite{KapurRokhlin1997,alp99,KolmRokhlin2001,KlocknerBarnettGreengardONeil2013, kre91, bregimrok10} or nonsmooth \cite{Kress1990, Helsing2011, Bremer2012, lulu14,
SerkhRokhlin2016} closed curves. We don't intend to discuss these methods in detail here but
wish to mention the analytical kernel splitting technique by Kress \cite{kre91}
for smooth curves. This method exploits the periodicity of the density and the
explicit formulae \cite{Martensen1963,Kussmaul1969} of Fourier coefficients of the
logarithmic singularity to achieve spectrally accurate discretizations of the
BIEs. Unfortunately, such a straightforward approach breaks down for
3D BIEs on smooth closed surfaces in general. Some typical high-order discretization techniques have been developed for this setting \cite{BrunoKunyansky2001,YingBirosZorin2006,BremerGimbutas2012,WalaKlockner2019,PerezArancibiaFariaTurc2019,PerezArancibiaTurcFaria2019,ZhuVeerapaneni2022}. Nevertheless,
none of these works takes advantage of the directional periodicity of the density on closed
surfaces. Consequently, a natural question is: can we extend Kress' kernel
splitting techniques to smooth closed surfaces using such a periodicity property?
We expect that the discretization of surface BIEs can be
effectively reduced to curve integrals, so that the aforementioned 2D techniques
apply.

In this sense, the above question is completely answered
for axisymmetric surfaces \cite{YoungHaoMartinsson2012,HelsingKarlsson2014,
EpsteinGreengardONeil2019,LaiONeil2019}. The density is periodic in the
azimuthal direction, while the kernels in the governing surface BIE are
convolutional. Therefore, by Fourier transforming in the azimuthal variable, the
surface BIE can be decomposed into a sequence of decoupled curve BIEs, each
involving a Fourier mode of the density and the corresponding Fourier
coefficient of the kernel. The required ${\cal O}(M)$ Fourier coefficients can
be evaluated in ${\cal O}(M\log M)$ operations \cite{LaiONeil2019} using
analytic kernel splitting, three-term recurrence relations evaluated by Miller's
algorithm \cite{Olver1964}, and FFT-accelerated discrete convolutions.
An essential feature that makes the above approach possible is that the squared
distance between source and target points is a function of the azimuthal
difference, containing only three Fourier modes $0$ and $\pm1$. Consequently, the
above approach breaks down in general for smooth surfaces of arbitrary shapes.

Motivated by the aforementioned situation, this paper is devoted to a novel
FFT-accelerated method for the high-order discretization of surface BIEs on general smooth surfaces. Like the axisymmetric case, since the unknown density is periodic in the azimuthal direction, Fourier
transforming the governing surface BIE in the azimuthal variable can still give
rise to curve integral operators, though coupled now. However, their kernels,
Fourier coefficients of the 3D kernels, are no longer straightforward to evaluate. To
tackle this issue, we derive, based on the assumption that
the squared distance has a short Fourier-expansion approximation in the azimuthal variable
(which is indeed true for analytic surfaces), higher-order recurrence relations for the
most singular part of the 3D kernels. Regarded as a boundary value
problem, they can be solved by Olver's algorithm \cite{Olver1964} with the
optimal ${\cal O}(M)$ complexity if ${\cal O}(M)$ Fourier coefficients are required. In
realistic applications with a moderately large $M$, we find a promising and much
faster approach. We use the singularity swapping approach
\cite{afKlintebergBarnett2021} to rewrite each singular kernel as the product of a
smooth function, analytic for analytic surfaces, and an axisymmetric-related
singular factor. Consequently, Miller's algorithm, for evaluating the Fourier
coefficients of the singular factor in ${\cal O}(M)$ operations, and the
FFT-accelerated discrete convolutions can be used to evaluate the ${\cal O}(M)$
Fourier coefficients in ${\cal O}(M\log M)$ operations, reducing surface
integrals to
an ${\cal O}(M)$ number of 2D curve integral operators. It turns out that this
new approach attains exactly the same order of complexity as for axisymmetric
surfaces! For analytic surfaces, we further prove that the resulting curve
integrals for  weakly singular  kernels exhibit the same logarithmic
singularities as in 2D BIEs \cite{kre91}. This allows us to use the panel-based
generalized Gaussian quadratures \cite{bregimrok10} to further discretize the
${\cal O}(M)$ curve integral operators, giving rise to a final linear system for
approximating the unknown density. Extensive numerical experiments are carried
out to illustrate the effectiveness of the new approach.

The rest of the paper is organized as follows. Section~\ref{sec2} introduces the
parametrization of smooth surfaces, and formulates 3D acoustic scattering
problems. Section~\ref{sec3} first establishes the logarithmic singular behavior
of weakly singular kernel functions after azimuthal integration and then
develops the high-order Nystr\"om discretization of the single-layer operator.
Section~\ref{sec4} extends the proposed discretization schemes to the
double-layer operator and its adjoint. Section~\ref{sec5} carries out several
numerical experiments demonstrating the effectiveness of the new method.
Section~\ref{sec6} concludes the whole paper.

\section{Problem formulation}\label{sec2}
Let $\Gamma$ be a closed smooth surface parameterized
by
\begin{align}  \label{eq:para}
  \br(t,\theta) = \{[x(t,\theta),y(t,\theta), z(t,\theta)]:(t,\theta)\in[0,L]\times [0,2\pi], L>0\},
\end{align}
satisfying the following assumptions:
\begin{enumerate}
\item[(i).] For any fixed $\theta\in[0,2\pi]$, $\br(t,\theta)$ is smooth and forms a simple (i.e., non-self-intersecting) curve for $t\in(0,L)$;
\item[(ii).] For any fixed $t\in[0,L]$, $\br(t,\theta)$ is smooth and
  $2\pi$-periodic for $\theta\in\mathbb{R}$;
  \item[(iii).] For any $(t,\theta)\in[0,L]\times[0,2\pi]$,
    $J(t,\theta)=|\br_t(t,\theta)\times\br_\theta(t,\theta)|>0$.
\end{enumerate}
Assumption (iii) ensures that the outer unit normal $\nu(\br)$ to $\Gamma$ is
well-defined everywhere. In what follows, we shall frequently assume that $\br$
is analytic.

Let $\Omega_e$ and $\Omega_i$ be the exterior and interior domains
separated by $\Gamma$. Let $\Delta = \partial_x^2 + \partial_y^2 + \partial_z^2$
be the 3D Laplacian, $g$ be a given function on
$\Gamma$, $k>0$ be a constant, and $r=|{\bm r}|$.
For $\Omega\in \{\Omega_e, \Omega_i\}$, the governing  equation is the 3D Helmholtz equation:
\begin{equation}
    \Delta u + k^2 u = 0,\quad{\rm in}\quad\Omega.
\end{equation}
In this paper, we focus only on two standard types of boundary conditions: the
Dirichlet boundary condition
\begin{align*}
u = g, \quad {\rm on}\quad \Gamma;
\end{align*}
and the Neumann boundary condition
\begin{align*}
\partial_{\nu}u = g, \quad {\rm on}\quad \Gamma;
\end{align*}
For the interior problem with $\Omega=\Omega_i$, we assume that $k$ is not an
eigenvalue for the corresponding boundary condition, whereas for the exterior
problem with $\Omega=\Omega_e$, we enforce the well-known Sommerfeld radiation
condition
\begin{equation}
\label{eq:src}    
\lim_{r\to\infty} r(\partial_r u - \bi k u) = 0,
\end{equation}
to ensure the well-posedness \cite{colkre13}.
In practice, the above problems can model acoustic waves scattering in $\Omega$. In the following, we shall use (IDP), (INP), (EDP) and (ENP) to denote the interior Dirichlet, interior Neumann, exterior Dirichlet, and exterior Neumann problems, respectively.

As is known, the fundamental solution of the Helmholtz operator $\Delta + k^2$
satisfying \eqref{eq:src} is
\begin{align}  \label{eq:helm:green}
  G(\br;\brs) &= \frac{e^{\bi k \rho(\br,\brs)}}{4\pi\rho(\br, \brs)},
\end{align}
where $\br = \br(t,\theta)$ and $\brs = \br(t_s,\theta_s)$ 
denote the target and source points, respectively, and the distance function
\begin{align}  \label{eq:def:rho}
  \rho(\br,\brs) &= |\br(t,\theta)-\br(t_s,\theta_s)|.
\end{align}
The above four problems can be formulated as BIEs through standard layer-potential representations and jump relations
\cite{colkre13}.
Assume for $\br\notin\Gamma$,
\begin{align}  \label{eq:u:sl}
  u(\br) &= 2\int_{\Gamma} G(\br;\brs)\phi(\brs)d \sigma(\brs),
\end{align}
for some unknown density function $\phi$ defined on $\Gamma$, where $d\sigma$ denotes
the differential of surface area. Then, both (IDP) and (EDP) can be restated as:
find $\phi$, such that
\begin{equation}  \label{eq:bie:S}
{\cal S}[\phi] = g, \quad{\rm on}\quad \Gamma,
\end{equation}
where ${\cal S}$ denotes the following single-layer potential operator,
\begin{align}  \label{eq:sl}
  {\cal S}[\phi](\br) &= 2\int_{\Gamma} G(\br;\brs)\phi(\brs)\dS(\brs),\quad \br\in\Gamma.
\end{align}
By the jump relations, (INP) and (ENP) are formulated as: find $\phi$,
such that
\begin{align}  \label{eq:bie:Kpie}
 ({\cal K}' \pm {\cal I})[\phi] = g, \quad{\rm on}\quad \Gamma,\quad \{+\ {\rm for\ (INP)},-\ {\rm for\ (ENP)}\},
\end{align}
where ${\cal I}$ denotes the identity operator and ${\cal K}'$ is the adjoint
double-layer potential operator
\begin{align}\label{eq:adl}
  {\cal K}'[\phi](\br) &= 2\cint_{\Gamma} \frac{\partial G(\br;\brs)}{\partial \bnu(\br)}\phi(\brs)\dS(\brs),\quad \br\in\Gamma,
\end{align}
and $\cint$ indicates Cauchy's principal value.

Alternatively, we may assume for $\br\notin\Gamma$,
\begin{align}  \label{eq:u:dl}
  u(\br) &= 2\int_{\Gamma} \frac{\partial G(\br;\brs)}{\partial \bnu(\brs)} \psi(\brs)\dS(\brs),
\end{align}
for some unknown density function $\psi$ defined on $\Gamma$. Then, (IDP) and
(EDP) are formulated by
\begin{align}  \label{eq:bie:Kie}
 ({\cal K} \mp {\cal I})[\psi] = g, \quad{\rm on}\quad \Gamma,\quad \{-\ {\rm for\ (IDP)},+\ {\rm for\ (EDP)}\},
\end{align}
where ${\cal K}$ denotes the double-layer potential operator
\begin{align}\label{eq:dl}
  {\cal K}[\psi](\br) &= 2\cint_{\Gamma} \frac{\partial G(\br;\brs)}{\partial \bnu(\brs)}\psi(\brs)\dS(\brs),\quad \br\in\Gamma.
\end{align}
Next, (INP) and (ENP) become
\begin{align}  \label{eq:bie:T}
  {\cal T}[\psi] = g, \quad{\rm on}\quad \Gamma,
\end{align}
where ${\cal T}$ denotes the hypersingular operator
\begin{align}\label{eq:hyp}
  {\cal T}[\psi](\br) &= 2\xint_{\Gamma} \frac{\partial^2 G(\br;\brs)}{\partial \bnu(\brs)\partial \bnu(\br)}\psi(\brs)\dS(\brs),\quad \br\in\Gamma,
\end{align}
where $\xint$ indicates Hadamard's finite part.

Compare the two formulations in equations~\eqref{eq:u:sl} and~\eqref{eq:u:dl}. For the two
Dirichlet problems (IDP) and (EDP), equation~\eqref{eq:u:dl} is preferred as the
resulting integral equations~\eqref{eq:bie:Kie} are
second-kind Fredholm equations. Nevertheless, equation~\eqref{eq:bie:S} is frequently
used in practice due to the simplicity of ${\cal S}$. Similarly,
equation~\eqref{eq:u:sl} is preferred for the two Neumann problems (INP) and (ENP) as the
integral equations~\eqref{eq:bie:Kpie} are second-kind
Fredholm equations. However, equation~\eqref{eq:bie:T} is rarely used since the hypersingular integral in equation~\eqref{eq:hyp} is more challenging to discretize.
In this paper, we shall develop fast and accurate quadrature rules to
discretize the weakly singular operators ${\cal S}$, ${\cal K}$ and ${\cal K}'$,
and shall defer the discretization of ${\cal T}$ in future work.

\section{Discretization of single-layer operator \texorpdfstring{${\cal S}$}{S}}\label{sec3}

Based on the parameterization in equation~\eqref{eq:para}, the single-layer operator becomes
\begin{align}  \label{eq:S:para}
  {\cal S}[\phi](t,\theta) = \int_{0}^{L}\int_{0}^{2\pi}\frac{e^{\bi k \rho(t,\theta,t_s,\theta_s)}}{2\pi \rho(t,\theta,t_s,\theta_s)} \tilde{\phi}(t_s,\theta_s)d\theta_s dt_s,
\end{align}
where $\tilde{\phi}(t,\theta) = \phi(t,\theta) J(t,\theta)$ and
the distance function $\rho(t,\theta,t_s,\theta_s) = \rho(\br,\br_s)$. 

\subsection{Logarithmic singularity and generalized Gauss quadrature}\label{sec31}
We first discuss the discretization of the integral w.r.t. $t_s$ in
equation~\eqref{eq:S:para}. In doing so, suppose
\begin{equation}  \label{eq:Sphi:para}
  S^\phi(t,t_s,\theta):=\int_{0}^{2\pi}\frac{e^{\bi k \rho(t,\theta,t_s,\theta_s)}}{2\pi
    \rho(t,\theta,t_s,\theta_s)} \tilde{\phi}(t_s,\theta_s)d\theta_s
\end{equation}
has been accurately approximated via certain quadratures that shall be presented
later, so that it suffices to approximate
\[  S[\phi](t,\theta)=\int_{0}^{L} S^\phi(t,t_s,\theta) d t_s.
\]

The following theorem describes the singularity for a family of singular kernels
after integrating out the azimuthal variable, including $S^\phi$ in
\eqref{eq:Sphi:para} and the related kernels for the double-layer operator
${\cal K}$ and  its adjoint ${\cal K}'$ that will be introduced in
Section~\ref{sec4}.

\begin{theorem}\label{thm:logsing}
Fix $t,\theta$ so that we write the distance $\rho(t_s,\theta_s)$ for short.
Consider 
\begin{equation}\label{eq:qm-general}
 q_m(t_s)=\int_0^{2\pi}[\rho^2(t_s,\theta_s)]^{-\frac m 2}
 f(t_s,\theta_s)\,d\theta_s,\quad m\in\{1,3,5,7,\cdots\},
\end{equation}
where $f$ is analytic and $2\pi$-periodic in $\theta_s$. Then, in a sufficiently small neighborhood $V$ of $t$, $q_m$ admits the representation
\begin{equation}\label{eq:log-splitting-general}
 q_m(t_s)=-\frac12 p_m(t_s)\log (t_s-t)^2+r_m(t_s),
\end{equation}
where $p_m$ is analytic and $r_m$ is meromorphic with $t_s=t$ as its only
possible pole. Moreover, if
$
 f(t_s,\theta_s)=\mathcal O\!\left(\rho^{m-1}(t_s,\theta_s)\right),
$
then $r_m$ is analytic in $V$.
\end{theorem}

\begin{proof}[Proof]
By Taylor's expansion of $\rho^2$ at $(t,\theta)$ and then using standard
perturbation theory, it can be seen that $\rho^2(t_s,\theta_s)$, for each fixed
$t_s$, has two zeros $\theta_\pm(t_s)$, closest to $\theta$, satisfying
\begin{equation}\label{eq:theta-pm-local}
\theta_\pm(t_s)-\theta =
\left( -\frac{\br_t\cdot\br_\theta}{|\br_\theta|^2} \pm \bi\frac{|\br_t\times\br_\theta|}{|\br_\theta|^2} \right)(t_s-t) + \mathcal O\!\left((t_s-t)^2\right),
\end{equation}
where we assume ${\rm Im}(\theta_+(t_s))>0$. Note that the two zeros are complex conjugates for real $t_s$, and coalesce for $t_s=t$. 

For a multivalued function $F(t_s)$, let $\widetilde F(t_s)$ denote the value
obtained by analytically continuing $t_s$ once counterclockwise around $t$,
and define the correction operator
\begin{equation*}
\mathcal C F(t_s):=\widetilde F(t_s)-F(t_s)
\end{equation*}
For example,  $\mathcal C\left[\log(t_s-t)^2\right](t_s)=4\pi \bi$. For $q_m(t_s)$, the original integration path, the line segment from
$0$ to $2\pi$, passing through the two singularities  $\theta_\pm(t_s)$, is denoted by $\Gamma_0$ in \eqref{eq:qm-general} as shown in Figure~\ref{fig:log-contours}(a). 

\begin{figure}[htbp]
\small
\centering

\begin{subfigure}[b]{0.32\textwidth}
\centering
\resizebox{\linewidth}{!}{%
\begin{tikzpicture}[>={Kite[inset=0, length=10, bend]},scale=0.8]
\begin{axis}[
xmin=-1, xmax=7, ymin=-3, ymax=3,
axis lines=center,
clip=false,
xlabel=$\Re\theta_s$,
ylabel=$\Im\theta_s$,
label shift=10pt,
axis equal image,
trig format plots=rad,
axis line style={
    -{Stealth[length=4pt,width=5pt]},
    line width=0.8pt
},
xtick={0,6.283185307179586},
xticklabels={$0$,$2\pi$},
hide obscured x ticks=false
]

\addplot[only marks, mark=*, blue] coordinates {
(3.14,1)(3.14,-1)
};
\node[right] at (3.14,1) {$\theta_+(t_s)$};
\node[right] at (3.14,-1) {$\theta_-(t_s)$};

\addplot[
black,
line width=1,
postaction={decorate},
decoration={
markings,
mark=at position 0.25 with {\arrow{>}},
mark=at position 0.75 with {\arrow{>}}
}
] coordinates {
(0,0)(6.28,0)
};
\node[above] at (5.25,0.05) {$\Gamma_0$};

\addplot[
red,
domain=0:2*pi,
samples=180,
line width=1,
postaction={decorate},
decoration={
markings,
mark=at position 0.1 with {\arrow{>}},
mark=at position 0.3 with {\arrow{>}},
mark=at position 0.5 with {\arrow{>}},
mark=at position 0.7 with {\arrow{>}},
mark=at position 0.9 with {\arrow{>}}
}
]
({x+2*sin(x)-sin(3*x)},{-2*sin(2*x)});

\node[red] at (5.15,2.25) {$\Gamma_1(t_s)$};

\end{axis}
\end{tikzpicture}%
}
\caption{\strut}
\end{subfigure}\hfill%
\begin{subfigure}[b]{0.32\textwidth}
\centering
\resizebox{\linewidth}{!}{%
\begin{tikzpicture}[>={Kite[inset=0, length=10, bend]},scale=0.8]
\begin{axis}[
xmin=-1, xmax=7, ymin=-3, ymax=3,
axis lines=center,
clip=false,
xlabel=$\Re\theta_s$,
ylabel=$\Im\theta_s$,
label shift=10pt,
axis equal image,
trig format plots=rad,
axis line style={
-{Stealth[length=4pt,width=5pt]},
line width=0.8pt
},
xtick={0,6.283185307179586},
xticklabels={$0$,$2\pi$},
hide obscured x ticks=false
]

\addplot[only marks, mark=*, blue] coordinates {
(3.14,1)(3.14,0)(3.14,-1)
};
\node[right] at (3.14,1) {$\theta_+(t_s)$};
\node[right] at (3.14,0) {$\theta_b$};
\node[right] at (3.14,-1) {$\theta_-(t_s)$};

\addplot[
black,
domain=0:pi,
samples=120,
line width=1,
postaction={decorate},
decoration={
markings,
mark=at position 0.10 with {\arrow{>}},
mark=at position 0.50 with {\arrow{>}},
mark=at position 0.999 with {\arrow{>}}
}
]
({3.14-sin(3*x)},{-2*sin(2*x)});

\addplot[
black,
domain=pi:2*pi,
samples=120,
line width=1,
postaction={decorate},
decoration={
markings,
mark=at position 0.50 with {\arrow{>}},
mark=at position 0.90 with {\arrow{>}}
}
]
({3.14-sin(3*x)},{-2*sin(2*x)});

\node at (5.05,2.25) {$D(t_s)$};

\end{axis}
\end{tikzpicture}%
}
\caption{\strut}
\end{subfigure}\hfill
\begin{subfigure}[b]{0.32\textwidth}
\centering
\resizebox{\linewidth}{!}{%
\begin{tikzpicture}[>={Kite[inset=0, length=10, bend]},scale=0.8]
\begin{axis}[
xmin=-1, xmax=7, ymin=-3, ymax=3,
axis lines=center,
clip=false,
xlabel=$\Re\theta_s$,
ylabel=$\Im\theta_s$,
label shift=10pt,
axis equal image,
trig format plots=rad,
axis line style={
    -{Stealth[length=4pt,width=5pt]},
    line width=0.8pt
},
xtick={0,6.283185307179586},
xticklabels={$0$,$2\pi$},
hide obscured x ticks=false
]

\addplot[only marks, mark=*, blue] coordinates {
(3.14,1)(2.14,0)(3.14,-1)
};
\node[right] at (3.14,1) {$\theta_+(t_s)$};
\node[right] at (2.2,0) {$\theta_b$};
\node[right] at (3.14,-1) {$\theta_-(t_s)$};

\addplot[
black,
domain=0:2*pi,
samples=120,
line width=1,
postaction={decorate},
decoration={
markings,
mark=at position 0.1 with {\arrow{>}},
mark=at position 0.4 with {\arrow{>}},
mark=at position 0.6 with {\arrow{>}},
mark=at position 0.9 with {\arrow{>}}
}
]
({3.14-cos(x)},{-2*sin(x)});

\node at (5.00,2.25) {$P(t_s)$};

\end{axis}
\end{tikzpicture}%
}
\caption{\strut}
\end{subfigure}

\vspace{-12pt}
\caption{The three integration paths: $\Gamma_1(t_s)$, and two closed curves $D(t_s)$ and $P(t_s)$, starting and ending at $\theta_b$.}
\label{fig:log-contours}
\vspace{-20pt}
\end{figure}
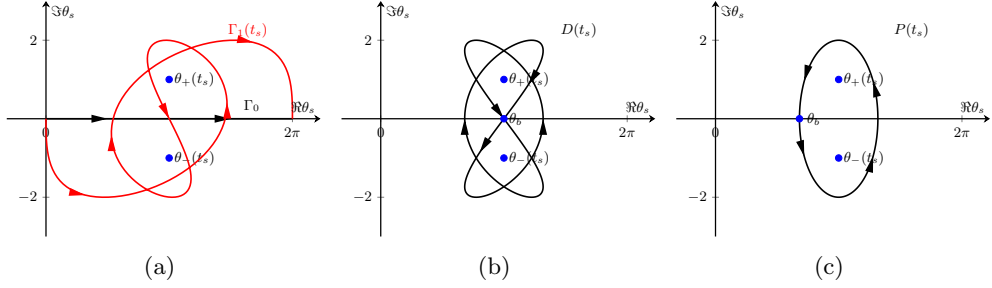

As $t_s$ winds once counterclockwise around $t$, 
$\theta_\pm(t_s)$ do likewise
by \eqref{eq:theta-pm-local}. To avoid crossing $\theta_\pm$,
the integration path must be deformed accordingly.
It becomes a path $\Gamma_1(t_s)$ that starts at $0$, positively encircles $\theta_-$ once, positively encircles $\theta_+$ once, then negatively encircles $\theta_-$ once, and finally negatively encircles $\theta_+$ once, ending at $2\pi$ (see Figure~\ref{fig:log-contours}(a)). Compared to  $\Gamma_0$, this deformation introduces an additional term:
\begin{align*}
\mathcal Cq_m(t_s)
=
\int_{\theta_b}^{(\theta_-+,\theta_++,\theta_--,\theta_+-)}
[\rho^2(t_s,\theta_s)]^{-\frac m 2}f(t_s,\theta_s)\,d\theta_s,
\end{align*}
which is the well-known Pochhammer double-loop contour integral (see Figure \ref{fig:log-contours}(b)). Here, 
fix  $\theta_b\in(0,2\pi)$ away from $\theta_\pm(t_s)$ and continue the principal branch along the contour.
Since the singularities at $\theta_\pm$ are of square root type, we have
\[
\int_{\theta_b}^{(\theta_-+,\theta_-+)}[\rho^2(t_s,\theta_s)]^{-\frac m 2}f(t_s,\theta_s)\,d\theta_s=\int_{\theta_b}^{(\theta_--,\theta_--)}[\rho^2(t_s,\theta_s)]^{-\frac m 2}f(t_s,\theta_s)\,d\theta_s=0.
\]
Therefore,
\begin{align*}
\mathcal Cq_m(t_s)
=
\int_{\theta_b}^{(\theta_-+,\theta_++,\theta_-+,\theta_++)}
[\rho^2]^{-\frac m 2}f(t_s,\theta_s)\,d\theta_s
=
2\int_{\theta_b}^{(\theta_-+,\theta_++)}
[\rho^2]^{-\frac m 2}f(t_s,\theta_s)\,d\theta_s,
\end{align*}
and the above simplified path is shown in Figure~\ref{fig:log-contours}(c). As $t_s\to t$ (i.e., $\theta_+\to\theta$), the fixed deformation moves away from the singularity. Hence, $\mathcal Cq_m(t_s)$ is analytic at $t_s = t$.
Define $p_m(t_s) := -\frac{1}{2\pi\bi}\mathcal
Cq_m(t_s)$, which is analytic in $V$. Then,
\begin{equation*}\label{eq:log-monodromy}
\mathcal C\left[-\frac12p_m(t_s)\log(t_s-t)^2\right](t_s)
=
-\frac12p_m(t_s)\mathcal C\left[\log(t_s-t)^2\right](t_s)
=\mathcal Cq_m(t_s).
\end{equation*}
Consequently, 
$
r_m(t_s)
=
q_m(t_s)
+\frac12p_m(t_s)\log(t_s-t)^2
$
is a single-valued and meromorphic function in $V$, which implies \eqref{eq:log-splitting-general}.

Moreover, if $ f(t_s,\theta_s)=\mathcal O\!\left(\rho^{m-1}\right)$, then
$[\rho^2(t_s,\theta_s)]^{-m/2}f=\mathcal O(\rho^{-1})$. As $\rho^{-1}$ is locally integrable in
$(t_s,\theta_s)$, $q_m(t_s)$ is locally integrable near $t$.
Since the logarithmic term in \eqref{eq:log-splitting-general} is also
locally integrable, $r_m$ cannot have a nonremovable pole at $t$.
Therefore $r_m$ is analytic in $V$.
\end{proof}

For the moment, we assume $\partial\Omega$ is not a torus in the sense that
$\br(t,\theta)$ is not periodic w.r.t $t$, i.e., $\br(L,\theta)\neq
\br(0,\theta)$. This ensures that the two curves $t=L$ and $t=0$ are
well-separated. 
Write
$
\frac{e^{\bi k\rho}}{\rho}
=\frac{\cos(k\rho)}{\rho}
+\bi\,\frac{\sin(k\rho)}{\rho}.
$
Since $\cos(k\rho)$ and $\sin(k\rho)/\rho$ are analytic functions of $\rho$ near $\rho=0$, the first term is covered by Theorem~\ref{thm:logsing} with $m=1$, while the second term is analytic. 
Thus, the only singularity of $S^\phi(t,t_s,\theta)$ in the $t_s$ variable occurs at $t_s=t$ and is logarithmic. We
adopt a similar quadrature developed in \cite[Section 5]{LaiONeil2019} that deals with
logarithmically singular integrals on generating curves for axisymmetric surfaces.
For completeness, we sketch the basic idea below.

First, the parameter domain $[0,L]$ is uniformly divided into $N_p$ panels $I_i
= [t_i,t_{i+1}]$ with $t_i = \frac{i-1}{N_p} L$ for $i=1,\cdots, N_p+1$. Inside each
panel $I_i$, $t$ is sampled at the $16$ Legendre points
$\{t_{i}^{l}\}_{l=1}^{16}$. Thus, the total number of sampling points for $t$ is
$16N_p$. Suppose $t=t_j^{l}$ in an interior panel $I_j$ for $j\in[2,N_p-1]$ and
$l\in [1,16]$; the case when $t$ belongs to the two boundary panels $I_1$ and
$I_{N_p}$ can be treated exactly in the same way. Thus, we decompose $S[\phi]$ into
three terms
\begin{align}  \label{eq:S123}
  S[\phi](t_j^{l},\theta) &= \left[ \sum_{|i-j|\geq 2}+ \sum_{|i-j|=1} + \sum_{i=j} \right]\int_{I_i} S^\phi(t_j^l,t_s,\theta)dt_s =: S_{1} + S_{2} + S_{3}.
\end{align}
In the above, $S_1$ is sufficiently smooth as $t_s$ lies in the panels that are
well separated from $I_j$. Thus, it is natural to directly use the $16$-point
Gauss-Legendre quadrature to approximate
\begin{equation}  \label{eq:S1}
  S_1 \approx \sum_{|i-j|\geq 2} \sum_{q=1}^{16} S^\phi(t_j^{l},t_i^{q},\theta)w_i^q,
\end{equation}
where $w_i^q$ are the corresponding weights in panel $i$.

As for $S_2$ and $S_3$ involving the adjacent panels $I_{j\pm 1}$ and the
self-interaction panel $I_j$, we adopt the generalized Gauss quadrature rule
\cite{bregimrok10} to discretize the integrals, considering the nearly singular
or weakly singular integrand $S^\phi(t_j^l,t_s,\theta)$. In our numerical
examples, we use the 16th-order generalized Gauss quadrature rule, which contains 16
support nodes and weights for both the adjacent and self-interaction panels;
note that the nodes and weights vary with $l$ in $t_j^l$ and none of the nodes
is equal to $t_j^l$. Here, we do not follow \cite{LaiONeil2019} that used 48 support
nodes and weights, though independent of $l$, for the adjacent panels $I_{j\pm
  1}$, as this requires three times the number of evaluations of
$S^{\phi}(t,t_s,\theta)$ using 16 moving nodes and weights. The nodes and
weights are produced based on the Fortran code at github.com/JamesCBremerJr/GGQ
and are evaluated for $t$ in one panel only, since the nodes and weights for $t$
in a different panel $I_{j'}$ can be obtained by simply applying a linear
transformation to those for $I_j$.
\begin{remark}
  For a torus-like surface with ${\bm r}(0,\theta)={\bm r}(L,\theta)$, the two
  boundary panels $I_1$ and $I_{N_p}$ become adjacent. Thus, $S_2$ in
  equation~\eqref{eq:S123} should contain the index $i$ with $|i-j|=N_p-1$, whereas $S_1$
  should not.
\end{remark}

\subsection{Fast discretization of \texorpdfstring{$S^\phi(t,t_s,\theta)$}{Sphi(t,ts,theta)}}\label{sec32}
In the following, we develop a fast Fourier transform (FFT) based
algorithm to rapidly discretize $S^\phi(t,t_s,\theta)$.

According to the previous
section, $t$ and $t_s$ are collocated at two disjoint sets of nodes,
i.e., it always holds that $t\neq t_s$ in the discretization.
When $t$ and $t_s$ are in well-separated panels, it is
safe and accurate to regard the integrand in equation~\eqref{eq:Sphi:para} as a sufficiently
smooth periodic function so that one directly uses the trapezoidal rule to
approximate
\begin{equation}  \label{eq:Sphi:sm}
  S^\phi(t,t_s,\theta) \approx \frac{1}{2M}\sum_{l=0}^{2M-1}\frac{e^{\bi k \rho(t,\theta,t_s,\theta_l)}}{\rho(t,\theta,t_s,\theta_l)}\tilde{\phi}(t_s,\theta_l),
\end{equation}
where $\theta_l = \frac{\pi l}{M}, l=0,\cdots, 2M-1$ for some positive integer
$M$ depending on the desired accuracy. When the surface
$\Gamma$ is analytic, the above approximation attains spectral accuracy of error
${\cal O}(e^{-c M})$ for some constant $c>0$.

When $t$ and $t_s$ are in adjacent panels or even in the same panel, the
sampling points for $t$ and the nodes for $t_s$ can be sufficiently close yet
not equal to each other. In this case, the integrand in equation~\eqref{eq:Sphi:para} is
close to singular at $\theta_s=\theta$. The consequence is that
\eqref{eq:Sphi:sm} requires an extremely large number $M$ of discretization
points to resolve the singularity. To tackle this difficulty, we make use of the
Fourier series of $\tphi$ to approximate the integral. Throughout this paper, we
denote for any $2\pi$-periodic function $f$, its $m$-th Fourier coefficient by
\[  \Fr{f}_m := \frac{1}{2\pi}\int_{0}^{2\pi} f(\theta)e^{-\bi m \theta}d\theta,
\]
for $m\in\mathbb{Z}$. For analytic $\Gamma$, $\tilde{\phi}$ is
expected to be $2\pi$-periodic and
analytic w.r.t. $\theta_s$. Therefore,  we can use a truncated Fourier series
of length $2M$ to approximate
\[  \tphi(t_s,\theta_s) = \sum^{\infty}_{m=-\infty}\Fr{\tphi}_m(t_s) e^{\bi m
    \theta_s} \approx \sum^{M-1}_{m=-M}\Fr{\tphi}_m(t_s) e^{\bi m
    \theta_s},
\]
where $M$ depends on a prescribed accuracy.
Thus,
\begin{align}  \label{eq:Sphi:appr1}
  {\cal S}^\phi(t,t_s,\theta) \approx \sum^{M-1}_{m=-M}\Fr{\tphi}_m(t_s)\Fr{\frac{e^{\bi k \rho}}{\rho}}_{-m}(t,\theta, t_s).
\end{align}
If $t$ and $t_s$ are in well-separated panels, the Fourier coefficients of
$\frac{e^{\bi k \rho}}{\rho}$ can be approximated by using the trapezoidal rule and
evaluated by the FFT algorithm; note that this in fact reproduces
equation~\eqref{eq:Sphi:sm} by the discrete Parseval's relation. Nevertheless, due to the
singularity of $\rho^{-1}$, the trapezoidal rule breaks down and hence FFT is
not applicable. We approximate
\begin{align}  \label{eq:Sphi:appr2}
  {\cal S}^\phi(t,t_s,\theta)&\approx\sum^{M-1}_{m=-M}\Fr{\tphi}_m(t_s)\sum^{M_0-1}_{n=-M_0}\Fr{\cos(k \rho)}_n(t,\theta,t_s)\Fr{\rho^{-1}}_{-(m+n)}(t,\theta, t_s)\nonumber\\
  &+\bi\sum^{M-1}_{m=-M}\Fr{\tphi}_m(t_s)\Fr{\frac{\sin(k \rho)}{\rho}}_{-m}(t,\theta, t_s).
\end{align}
In the above, we use the truncated Fourier series of length $2M_0$ to
approximate the analytic function $\cos(k\rho)$. In practice, $M_0\leq M$ and
depends only on $k$ and the surface $\Gamma$; in our numerical examples, we choose $M_0=M$ to avoid possible loss of accuracy. To
proceed, we encounter an essential question: how to efficiently evaluate the
Fourier coefficients $\Fr{\rho^{-1}}_{j}$ for $|j|\leq M+M_0$? Due to the highly
oscillatory term $e^{\bi(m+n)\theta_s}$, we require at least $O(m+n)$ points to
accurately evaluate each $\Fr{\rho^{-1}}_{-(m+n)}$. Therefore, it totally
requires ${\cal O}(M^2)$ operations to evaluate all the Fourier coefficients for
some fixed $t$, $\theta$, and $t_s$ with $t\neq t_s$. In the following, we
propose two fast algorithms to evaluate the ${\cal O}(M)$ Fourier coefficients of $\rho^{-1}$. 
\subsubsection{An ${\cal O}(M)$ approach}
\label{subsubsec:linear}
Since $\rho^2(t,\theta,t_s,\theta_s)$ is
$2\pi$-periodic w.r.t. $\theta_s$, we have
\[  \rho^2(t,\theta,t_s,\theta_s) = \sum_{n=-\infty}^{+\infty}\Fr{\rho^2}_n(t,\theta,t_s) e^{\bi n\theta_s}.
\]
For $t\neq t_s$,
\begin{align*}  0 &= \int_{0}^{2\pi} \frac{d}{d\theta_s}(2e^{\bi m\theta_s}\rho)d\theta_s 
  =\int_{0}^{2\pi}  \frac{e^{\bi m\theta_s}}{\rho}\left(2 \bi m\rho^2 + (\rho^2)'\right)d\theta_s\\
  &=\bi\sum_{n=-\infty}^{+\infty}\left(2 m + n\right)\Fr{\rho^2}_n\int_{0}^{2\pi}  \frac{e^{\bi (m+n)\theta_s}}{\rho}d\theta_s
  =2\pi\bi\sum_{n=-\infty}^{+\infty}(2m+n)\Fr{\rho^2}_n\Fr{\rho^{-1}}_{-(m+n)}.
\end{align*}
We thus get the following lemma.

\begin{lemma}
  Let $t\neq t_s$. For any $m\in\mathbb{Z}$,
  \begin{align}    \label{eq:iter:formula}
    \sum_{n=-\infty}^{+\infty}(2m+n)\Fr{\rho^2}_n(t,\theta,t_s)\Fr{\rho^{-1}}_{-(m+n)}(t,\theta,t_s) &= 0.
\end{align}
\end{lemma}
Since $\rho^2$ is analytic with respect to $\theta_s$, $\Fr{\rho^2}_n$
decays exponentially as $|n|$ increases. Therefore, we may truncate the series
in equation~\eqref{eq:iter:formula} as
\begin{equation}  \label{eq:ld}
\sum_{n=-m_0}^{m_0}(2m+n)\Fr{\rho^2}_n(t,\theta,t_s)\Fr{\rho^{-1}}_{-(m+n)}(t,\theta,t_s) \approx 0,
\end{equation}
for a properly chosen integer $m_0$. In practice, $m_0$ is chosen such that, for
$(t,\theta,t_s)$ ranging over all grid points of $\Gamma$,
\begin{equation}  \label{eq:crt:M0}
\frac{|\Fr{\rho^2}_{\pm m_0}(t,\theta,t_s)|}{\max_{|n|\leq
    m_0}|\Fr{\rho^2}_{n}(t,\theta,t_s)|}\leq \epsilon_0,
\end{equation}
for a specified error threshold $ \epsilon_0$ (e.g., $\epsilon_0=10^{-14}$). Clearly, $m_0$
depends only on $\Gamma$ and we expect $m_0\ll M$. The $2m_0+1$-term recurrence relation in equation~\eqref{eq:ld} serves as the key for effectively evaluating the Fourier coefficients $\{\Fr{\rho^{-1}}_{j}\}$ for $|j|\leq M+M_0$.

As $\Fr{\rho^{-1}}_{-l}=\overline{\Fr{\rho^{-1}}_{l}}$, it suffices to
evaluate $\Fr{\rho^{-1}}_l$, $l=0,\cdots, M+M_0$. By equation~\eqref{eq:ld}, these unknowns are governed by
\begin{equation}  \label{eq:lin:diff}
  \sum_{j=0}^{2m_0} a_j(i) \Fr{\rho^{-1}}_{i+j} = 0,\quad i=0,\cdots, M_1,
\end{equation}
where $M_1=M+M_0-2m_0$ and the coefficients $\{a_j(i)\}_{j=0}^{2m_0}$ are defined as
\begin{align}  \label{eq:def:a_l}
  a_j(i) = (m_0+2i+j)\Fr{\rho^2}_{m_0-j}.
\end{align}
 The linear difference equation gives rise to $M_1+1$ linear equations governing the
$M+M_0+1$ unknowns $\Fr{\rho^{-1}}_j$ for $j=0,\cdots, M+M_0$. Thus, to solve
this linear system of equations, we have to add $2m_0$ extra conditions. One
simple approach is to precompute the first $m_0$ terms and the last $m_0$ terms, i.e.,
\begin{align}  \label{eq:offline:terms}
  \Fr{\rho^{-1}}_0,\cdots, \Fr{\rho^{-1}}_{m_0-1}\ {\rm and}\ \Fr{\rho^{-1}}_{M+M_0-m_0+1},\cdots, \Fr{\rho^{-1}}_{M+M_0}.
\end{align}
In the implementation, we directly apply the adaptive Gauss-Legendre quadrature to compute these coefficients.
Then, equation~\eqref{eq:lin:diff} gives rise to the following linear system
\begin{equation}  \label{eq:Axb}
    {\bf A} {\bf x} = {\bf b},
\end{equation}
  where ${\bf A}$ is an $(M_1+1)\times (M_1+1)$ matrix
  with elements
  \[
    A_{ij} = \left\{
      \begin{array}{ll}
      a_{j-i+m_0}(i) & {\rm if}\quad|j-i|\leq m_0;\\
      0 & {\rm otherwise},
      \end{array}
  \right.\quad i,j=0,\cdots M_1,
  \]
the $M_1+1$ column vector ${\bf b}$ is related to the
  precomputed terms in equation~\eqref{eq:offline:terms}, and the unknowns
  \begin{equation}    \label{eq:root:x}
    {\bf x} = [\Fr{\rho^{-1}}_{m_0},\cdots, \Fr{\rho^{-1}}_{M+M_0-m_0}]^{T}.
\end{equation}
  By solving the linear system in equation~\eqref{eq:Axb}, we obtain all required Fourier coefficients.
  It can be seen that the above approach requires only ${\cal O}(M)$ operations to evaluate the Fourier coefficients $\Fr{\rho^{-1}}_{-(m+n)}$ for any fixed $t$, $\theta$, and $t_s$ with $t\neq t_s$.

  The following lemma provides a sufficient condition for the solvability of equation~\eqref{eq:Axb}.
\begin{lemma}
  The $2m_0+1$-banded coefficient matrix ${\bf A}$ is Hermitian. If $\rho^2$ meets the following condition
  \begin{equation}    \label{eq:cond:rho}
    2\Fr{\rho^2}_0 > \sum_{l=-\infty}^{\infty}\left|\Fr{\rho^2}_l  \right|,
\end{equation}
  then ${\bf A}$ is strictly positive definite and hence nonsingular.
  \begin{proof}
    For $|j-i|\leq m_0$,
    \begin{align*}      A_{ji} &= a_{i-j+m_0}(j) =(m_0+2j+(i-j+m_0)) \Fr{\rho^2}_{m_0-(i-j+m_0)}\\
      &= (2m_0+i+j)\Fr{\rho^2}_{j-i} = \overline{A_{ij}},
\end{align*}
    so that ${\bf A}$ is Hermitian. On the other hand, for $i=0,\cdots, M_1$,
    \begin{align*}      &\sum_{j=i-m_0}^{i+m_0} |a_{j-i+m_0}(i)| - 2 |a_{m_0}(i)|= \sum_{j=i-m_0}^{i+m_0} (2m_0+i+j)|\Fr{\rho^2}_{i-j}| - 2 (2m_0+2i)|\Fr{\rho^2}_0|\\
      &=\sum_{j=-m_0}^{m_0} (2m_0+2i+j)|\Fr{\rho^2}_{j}| - 2 (2m_0+2i)\Fr{\rho^2}_0\\
      &=(2m_0+2i)\left[\sum_{j=-m_0}^{m_0} |\Fr{\rho^2}_{j}| -  2\Fr{\rho^2}_0  \right] < 0.
\end{align*}
    Thus, ${\bf A}$ is diagonally dominant and hence strictly positive definite,
    considering that all diagonal elements of ${\bf A}$ are positive.
  \end{proof}
\end{lemma}
\begin{remark}
For axisymmetric surfaces, \cite{LaiONeil2019} used a three-term recurrence
relation to evaluate the Fourier coefficients of $\rho^{-1}$. This can be simply
explained via equation~\eqref{eq:iter:formula}. Specifically, for an axisymmetric
surface $\Gamma$, we may parameterize it by
\[  \br(t,\theta) = \{[r(t)\cos\theta, r(t)\sin\theta, z(t)]: 0\leq t\leq L, 0\leq
  \theta\leq 2\pi\}.
\]
Then, it is easy to derive that
\begin{align*}  \rho^2(t,\theta,t_s,\theta_s) &= r^2(t) + r^2(t_s) + (z(t) - z(t_s))^2 - 2r(t)r(t_s)\cos(\theta-\theta_s),
\end{align*}
indicating that $\Fr{\rho^2}_n\equiv 0$ for $|n|\geq 2$ and hence $m_0=1$. Clearly, $\rho^2$ satisfies \eqref{eq:cond:rho}. Since $\rho^2$ depends only on $\cos(\theta-\theta_s)$, 
\begin{equation}  \label{eq:axial}
  \Fr{\rho^{-1}}_m(t,\theta,t_s) = e^{\bi m\theta}\Fr{\rho^{-1}}_m(t,0,t_s).
\end{equation}
Thus, it suffices to evaluate Fourier coefficients of $\rho^{-1}$ for
$\theta=0$. The above demonstrates the essential feature for the FFT-accelerated
algorithm in \cite{LaiONeil2019}.
\end{remark}

\subsection{\texorpdfstring{An $\mathcal{O}(M\log M)$ approach}{An O(M log M) approach}}
\label{subsubsec:OMlogM}
If $M$ is not extremely large, we find the following singularity
swapping approach more promising for $t$ close to $t_s$
\cite{afKlintebergBarnett2021}. The key idea is to  extract the nearest complex
singularity directly from the geometry and then evaluate all Fourier
coefficients by one FFT convolution. Fix $t,\theta,t_s$, and write for short
\begin{equation}\rho^2(\theta_s)= \left[ x(t,\theta)-x(t_s,\theta_s) \right]^2
+ \left[ y(t,\theta)-y(t_s,\theta_s) \right]^2
+ \left[ z(t,\theta)-z(t_s,\theta_s) \right]^2.
\label{eq:rho}
\end{equation}
Recall $\theta_\pm=a\pm\ii b, b>0$, defined in \eqref{eq:theta-pm-local}. $\theta_+$ is
numerically located by Newton's method with the initial guess
$\theta_+^0=a_0+\ii b_0$, where $a_0$ is a real minimizer of $\rho^2$ and
$b_0=\sqrt{2\rho^2(a_0)/[\rho^2]''(a_0)}$.
The following lemma extracts the most singular factor of $\rho^2(\theta_s)$.

\begin{lemma}
\label{singular1}
Suppose ${\bm r}(t,\theta)$ is analytic in $\theta$, and the roots $\theta_\pm =
a\pm \bi b$
of $\rho^2$ are simple.
Let
\begin{equation}s_{a,b}(\zeta)
=
1+e^{-2b}-2e^{-b}\cos(\zeta-a).
\label{eq:sab}
\end{equation}
Then $s_{a,b}(\zeta)$ has the same pair of simple roots, and
\begin{equation}F(\zeta)
=
\left(\frac{s_{a,b}(\zeta)}{\rho^2(\zeta)}\right)^{1/2}
\label{eq:F-smooth}
\end{equation}
becomes holomorphic at $\zeta=a\pm \bi b$.
\begin{proof}
Since
$
s_{a,b}(\zeta)
=
\bigl(1-e^{\ii(\zeta-a)-b}\bigr)
\bigl(1-e^{-\ii(\zeta-a)-b}\bigr),
$
$\theta_\pm$ are exactly its two simple roots. Near
$\theta_+$,
$
\rho^2(\zeta)=(\zeta-\theta_+)\rho^2_1(\zeta),
s_{a,b}(\zeta)=(\zeta-\theta_+)s_1(\zeta),
$
with $\rho^2_1(\theta_+)s_1(\theta_+)\neq0$. Hence $s_{a,b}/\rho^2=s_1/\rho^2_1$
is analytic and nonzero at $\theta_+$, and similarly at ${\theta_-}$. Therefore,
$F(\zeta)$ admits an analytic continuation across both zeros.
\end{proof}
\end{lemma}
Lemma~\ref{singular1} indicates that $F$ can be accurately represented by a
short Fourier expansion
\begin{equation}F(\theta_s)
\approx
\sum_{j=-p}^{p}\FC{F}_j e^{\ii j\theta_s},
\label{eq:F-expansion}
\end{equation}
where $p$ can be chosen by a similar rule for $m_0$ according to
\eqref{eq:crt:M0},
and is independent of $M$.
As $\rho^{-1}=F/\sqrt{s_{a,b}}$, \eqref{eq:F-expansion} yields
\begin{align}\FC{\rho^{-1}}_l
&=\frac{1}{2\pi}\int_0^{2\pi}
\frac{F(\theta_s)e^{-\ii l\theta_s}}
{\sqrt{s_{a,b}(\theta_s)}}\,d\theta_s \nonumber\\
&\approx \sum_{j=-p}^{p}\FC{F}_j\,
\frac{1}{2\pi}\int_0^{2\pi}
\frac{e^{-\ii(l-j)\theta_s}}
{\sqrt{1+e^{-2b}-2e^{-b}\cos(\theta_s-a)}}\,d\theta_s \nonumber\\
&=\sum_{j=-p}^{p}\FC{F}_j\,
e^{-\ii(l-j)a}A_{|l-j|}(b),\qquad l=0,\cdots, M+M_0,
\label{eq:convolution}
\end{align}
where the change of variable $\phi=\theta_s-a$ is used in the last equality and
\begin{equation}
A_m(b):=\FC{s_{0,b}^{-1/2}}_m = \frac{1}{2\pi}\int_0^{2\pi}
\frac{e^{-\ii m\phi}}{\sqrt{1+e^{-2b}-2e^{-b}\cos\phi}}\,d\phi,
\label{eq:Am}
\end{equation}

According to \cite{LaiONeil2019},
$A_m$ are scaled half-integer Legendre functions,
\begin{equation}A_m(b)=\frac{e^{b/2}}{\pi}Q_{m-1/2}(\cosh b),
\label{eq:Am-Q}
\end{equation}
and can be evaluated using the stable ${\cal O}(M)$ recurrence algorithms in \cite{LaiONeil2019}. For $b\ll 1$, the differences $D_m=A_m-A_{m-1}$ satisfy
\begin{equation}D_{m+1}=\frac{2m-1}{2m+1}D_m+
\frac{8m\sinh^2(b/2)}{2m+1}A_m.
\label{eq:diff-recurrence}
\end{equation}
We find that it is numerically more stable to evaluate $D_m$ first and update $A_m$ next. Consequently, the discrete convolution in \eqref{eq:convolution} for computing $\FC{\rho^{-1}}_l, l=0,\cdots, M+M_0$ can be performed by one zero-padded FFT convolution in ${\cal O}(M\log M)$ operations.

Figure~\ref{AGM}
\begin{figure}[htb]
    \centering
    (a)\includegraphics[width=0.45\textwidth]{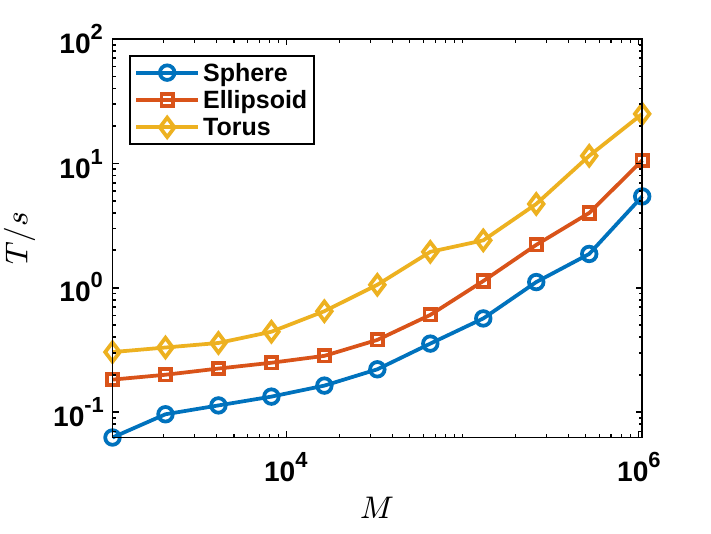}
    (b)\includegraphics[width=0.45\textwidth]{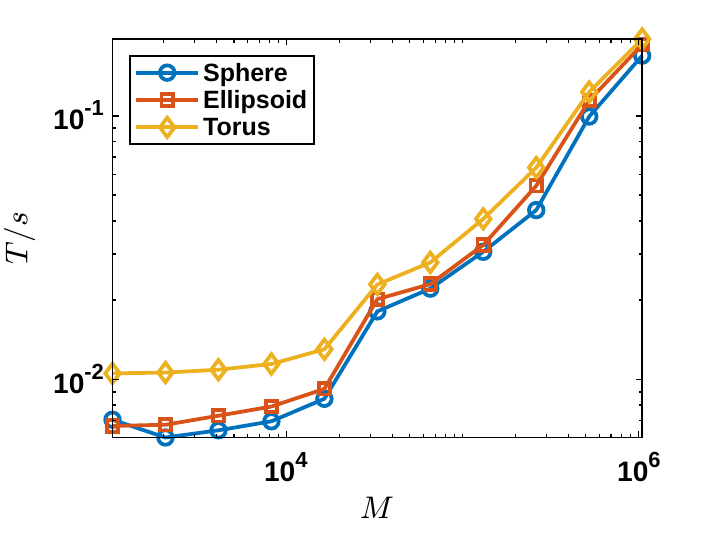}
   \caption{ CPU time $T$ for computing $\{\FC{\rho^{-1}}_l\}_{l=0}^{M}$ for a
   sphere, an ellipsoid, and a  twisted torus: (a) the ${\cal O}(M)$ approach; (b)
   the ${\cal O}(M\log M)$ approach. }
    \label{AGM}
\end{figure}
reports the CPU time $T$ for computing $\{\FC{\rho^{-1}}_l\}_{l=0}^{M}$ using the
two aforementioned approaches to attain the same level of accuracy, for an
axisymmetric unit sphere and two nonaxisymmetric geometries: an ellipsoid and a
twisted torus studied in Examples~1 and~4 of Section~\ref{sec5}, respectively.
Throughout this paper, all algorithms are implemented in MATLAB R2025b on a
machine equipped with an Intel Xeon CPU of 64 cores and 512 GB of RAM. In the
${\cal O}(M\log M)$ approach, we observe nearly identical running times for the
three different geometries; more importantly, the loss of axisymmetry introduces
no appreciable additional cost. Comparing the two different approaches, we find
the ${\cal O}(M\log M)$ approach is much faster even for $M\gg 1$ and we do not
intend to show the cross-over value of $M$ for the two approaches.

In practice, $M$ can be moderately large (around $10^2$ in our numerical
experiments). The previous ${\cal O}(M)$ approach, though optimal, possesses an
extremely large prefactor for numerically evaluating the highly oscillatory
integrals for the right tail indices; we wish to tackle this issue in the
future. Besides, as we shall see below, due to the convolutional form
\eqref{eq:Sphi:appr2}, the final computational complexity for discretizing
${\cal S}$ will always involve an ${\cal O}(M\log M)$ complexity term so that
this new ${\cal O}(M\log M)$ approach does not hurt the overall complexity at
all. 

\subsection{Final linear system}
With the above quadratures, the single-layer potential ${\cal S}[\phi]$ is fully
discretized in terms of the unknowns $\{\Fr{\tphi}_m(t_s)\}_{m=-M}^{M-1}$ when $t$ and $t_s$
are close enough. By the discrete Parseval's relation, the discrete inner
products in the Fourier space in equation~\eqref{eq:Sphi:appr2} are transformed to those
in the physical space so as to approximate $S^\phi$ in terms of
$\{\tphi(t_s,\theta_l)\}_{l=0}^{2M-1}$. To solve equation~\eqref{eq:bie:S}, one further
uses the method of interpolation to approximate $\tphi(t_s,\theta_l)$, say
$t_s\in I_i$, in terms of $\tphi(t_i^j,\theta_l)$ for $1\leq j\leq 16$, where we
recall $t_i^j$ are the Legendre points of panel $I_i$. By collocating ${\br}$ at
$\br(t_i^j,\theta_l)$ for $ 1\leq i\leq N_p, 1\leq j\leq 16, 0\leq l\leq 2M-1$,
Equation~\eqref{eq:bie:S} is discretized as
\begin{equation}  {\bm S}{\bm \phi} = {\bm g},
\end{equation}
where ${\bm \phi}$ and ${\bm g}$ denote the vectors of $\tphi$ and $g$ at the
$N=32MN_p$ points and ${\bm S}$ is a $N\times N$ matrix. Once $\phi$ is
approximated, we simply apply the 16-point Gauss-Legendre quadrature rule and
the trapezoidal rule to discretize the surface integral over $\Gamma$ in
equation~\eqref{eq:u:sl} to get $u({\bm r})$ for ${\bm r}\in\Omega$; in our numerical
computations, we always assume that ${\bm r}$ is sufficiently far away from the
surface $\Gamma$. For the close-to-surface evaluation, we refer readers to
\cite{BaoHuaLaiZhang2024,ZhuVeerapaneni2022} for more details.

To conclude this section, we discuss the complexity of constructing ${\bm S}$.
According to the number of panels $N_p$, ${\bm S}$ can be decomposed into $N_p\times
N_p$ block matrices: ${\bm S}_{ij}$ of size $(32M)\times (32M)$, relating to $t$
in panel $I_i$ and $t_s$ in panel $I_j$ in equation~\eqref{eq:S:para}, for $1\leq i,j\leq
N_p$. The whole procedure is divided into two stages: a preparation stage and a
correction stage. The preparation stage applies equations~\eqref{eq:S1} and
\eqref{eq:Sphi:sm} to construct each ${\bm S}_{ij}$ in ${\cal O}(M^2)$
operations for all $1\leq i,j\leq N_p$. The correction stage corrects entries of
each ${\bm S}_{ij}$ for $|i-j|\leq 1$ based on the faster algorithm for computing
$\{\Fr{\rho^{-1}}_{m}\}_{m=0}^{M+M_0}$ and the Fourier-to-physical
transformation by Parseval's relation. It is not hard to see that the complexity
of constructing each ${\bm S}_{ij}$ is ${\cal O}(M^2\log M)$. Consequently, the
total complexity of constructing ${\bm S}$ is ${\cal O}(M^2N_p\log M) + {\cal
O}(M^2N_p^2)$. In practice, $N_p\ll M$ and the entries in the preparation stage
are explicitly provided and can be initialized in a parallel way so that the
correction stage is in fact more time-consuming. Figure~\ref{pictimeS}
\begin{figure}[htb]
    \centering
    (a)\includegraphics[width=0.45\textwidth]{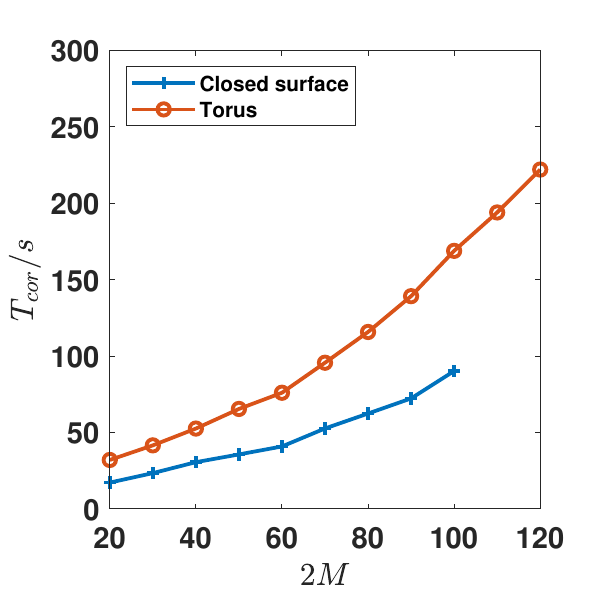}
    (b)\includegraphics[width=0.45\textwidth]{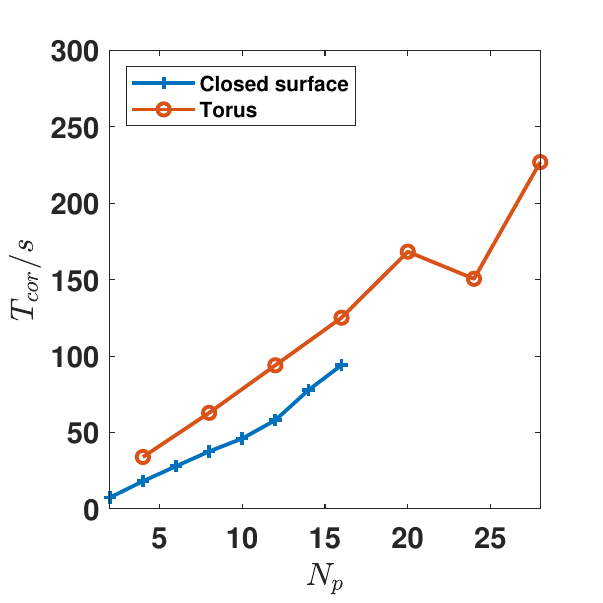}
    \caption{$T_{cor}$ against $2M$ and $N_p$: (a) $N_p$ is $28$ for 'o's and $16$ for '$+$'s; (b)  $2M$ is $120$ for 'o's and 100 for '$+$'s.}
    \label{pictimeS}
\end{figure}
 shows the running time $T_{cor}$ in the correction stage for different values of $2M$ and $N_p$ for two types of surfaces. They are a closed surface and a torus-like surface as studied in Examples 1 and 3 in Section~\ref{sec5}. The plots illustrate that $T_{cor}$ increases linearly with the number of panels $N_p$ and slightly superlinearly with $2M$.

\section{Discretization of ${\cal K}$ and its adjoint ${\cal K}'$ }\label{sec4}
We now turn to the discretization of ${\cal K}$ and ${\cal K}'$. Based on the
parameterization in equation~\eqref{eq:para},
the double-layer operator  in
equation~\eqref{eq:dl} and its adjoint ${\cal K}'$ in
equation~\eqref{eq:adl} can be written as
\begin{align*}{\cal K}[\psi](t,\theta)
&=
\int_{0}^{L}dt_s \int_0^{2\pi}
\frac{(\alpha(\rho)\tau_1
-\tau_3)(t,\theta,t_s,\theta_s)
-\bi\beta(\rho)\tau_0(t,\theta,t_s,\theta_s)}{2\pi}
\tilde{\psi}(t_s,\theta_s)\,d\theta_s,\\
{\cal K}'[\phi](t,\theta)
&=
\int_{0}^{L}dt_s \int_0^{2\pi}
\frac{(\alpha(\rho)\tau_1
-\tau_3)(t_s,\theta_s,t,\theta)
-\bi\beta(\rho)\tau_0(t_s,\theta_s,t,\theta)}{2\pi}
\tilde{\phi}(t_s,\theta_s)\,d\theta_s,
\end{align*}
where two analytic functions $\alpha$ and $\beta$ are defined by
\[
\alpha(\rho) = \frac{1-\cos(k\rho)-k\rho\sin(k\rho)}{\rho^2}, \quad \beta(\rho) = \frac{\sin(k\rho)-k\rho\cos(k\rho)}{\rho^3},
\]
$\tpsi(t,\theta)=\psi(t,\theta)J(t,\theta)$, and
\begin{align*}\tau_j(t,\theta,t_s,\theta_s)
&=
\frac{\nu(t_s,\theta_s)\cdot
\bigl(\br(t_s,\theta_s)-\br(t,\theta)\bigr)}
{\rho^j(t,\theta,t_s,\theta_s)},
\qquad j\in\{0,1,3\}.
\end{align*}
Due to the similarity between the two operators, we shall only present the
discretization of ${\cal K}$ in the following.

It thus suffices to
discretize
\[K^\psi(t,t_s,\theta)
:=
\int_0^{2\pi}\frac{1}{2\pi}
\left[
(\alpha(\rho)\tau_1
-\tau_3)(t,\theta,t_s,\theta_s)
-\bi\beta(\rho)\tau_0(t,\theta,t_s,\theta_s)
\right]
\tilde{\psi}(t_s,\theta_s)\,d\theta_s,
\]
with $t$ and $t_s$ located in two adjacent panels or in the same panel. For smooth surfaces, Taylor expansion at $(t,\theta)$ gives
\begin{equation}
    \label{eq:tau0:cond}
\tau_0(t,\theta,t_s,\theta_s)=\mathcal O(\rho^2) \qquad\text{as }(t_s,\theta_s)\to(t,\theta).
\end{equation}
The nonsmooth terms $\tau_1$ and $\tau_3$ in ${\cal K}$ are therefore of the
form~\eqref{eq:qm-general} with $m=1, 3$, with the analytic numerator $\tau_0$
satisfying \eqref{eq:tau0:cond}. Consequently, Theorem~\ref{thm:logsing} implies
that, after azimuthal integration, $K^\psi(t,t_s,\theta)$ has a logarithmic
singularity at $t_s=t$. Therefore, the previous generalized Gauss quadrature in
Section \ref{sec31} becomes applicable for the integral of $K^\psi$ w.r.t. $t_s$
from $0$ to $L$. We focus on the discretization of $K^\psi$ in the following.

As in Section~\ref{sec3}, we approximate $\tilde{\psi}$ by its truncated Fourier series of
length $2M$. Since $\alpha(\rho)$ and $\beta(\rho)\tau_0$ are analytic and
$2\pi$-periodic with respect to $\theta_s$, their Fourier expansions can be
truncated to a modest number, $2M_0$, of modes. We then obtain
\begin{align}K^\psi(t,t_s,\theta)
&\approx
\sum_{m=-M}^{M-1}\sum_{n=-M_0}^{M_0-1}
\Fr{\tilde{\psi}}_m(t_s)
\Fr{\alpha(\rho)}_n(t,\theta,t_s)
\Fr{\tau_1}_{-(m+n)}(t,\theta,t_s)
\nonumber\\
&-
\sum_{m=-M}^{M-1}
\Fr{\tilde{\psi}}_m(t_s)
\Fr{\tau_3+\bi {\beta(\rho)\tau_0}}_{-m}(t,\theta,t_s),
\label{eq:dis:K}
\end{align}
for integers $M$ and $M_0$ depending on the required accuracy. Thus, we need to
evaluate $\Fr{\tau_1}_{l}$ and $\Fr{\tau_3}_{l}$ for $l=-(M+M_0-2),\cdots, M+M_0$, as
discussed in the following.
As $\tau_j$ is real, $\Fr{\tau_j}_l = \overline{\Fr{\tau_j}_{-l}}$ so that we only
need to evaluate $\Fr{\tau_j}_l$ for $l=0,\cdots, M+M_0$.

For $\tau_1$, noticing that $\tau_0$ is analytic and
$2\pi$-periodic w.r.t. $\theta_s$, we have
\begin{align}  \label{eq:tau1}
  \Fr{\tau_1}_{l}(t,\theta,t_s) &= \sum_{n=-\infty}^{\infty} \Fr{\tau_0}_n(t,\theta,t_s)\Fr{\rho^{-1}}_{l-n}(t,\theta,t_s)\nonumber\\
  &\approx\sum_{n=-M_0}^{M_0} \Fr{\tau_0}_n(t,\theta,t_s)\Fr{\rho^{-1}}_{l-n}(t,\theta,t_s),\quad l=0,\cdots, M+M_0.
\end{align}

For $\tau_3$, we have
\begin{align}  \label{eq:tau3}
  \Fr{\tau_1}_{l}(t,\theta,t_s) &= \sum_{n=-\infty}^{\infty} \Fr{\tau_3}_{l-n}(t,\theta,t_s)\Fr{\rho^{2}}_{n}(t,\theta,t_s)\nonumber\\
  &\approx\sum_{n=-M_0}^{M_0} \Fr{\tau_3}_{l-n}(t,\theta,t_s)\Fr{\rho^{2}}_{n}(t,\theta,t_s),\quad l\in\mathbb{Z}.
\end{align}
One could follow the same ${\cal O}(M)$ approach in Section \ref{subsubsec:linear} to evaluate all
$\{\Fr{\tau_3}\}_{l=0}^{M+M_0}$ with the optimal ${\cal O}(M)$ complexity. However, as in Section \ref{subsubsec:OMlogM}, we introduce a numerically faster but theoretically ${\cal O}(M\log M)$ approach for evaluating $\FC{\tau_3}_l$ for a moderately large $M$.

Fix $t,\theta,t_s$ again, and recall the definitions of
$\rho^2(\theta_s)$, the nearest root $\theta_+=a+\ii b$ of $\rho^2$ in the upper
half-plane, $s_{a,b}(\theta_s)$, $F(\theta_s)$, and $A_m(b)$ in
Section~\ref{subsubsec:OMlogM}. Set
\begin{align}
\label{eq:H3}
H_3(\theta_s)&=\tau_0(\theta_s)F^3(\theta_s).
\end{align}
By Lemma~\ref{singular1}, $H_3(\theta_s)$ is analytic in the neighborhood of $\theta_s=a\pm \bi b$ as $\tau_0$ is analytic.

We consider $b$ sufficiently far away from $0$ first.  As $\tau_3 = {\tau_0}{(\rho^2)^{-3/2}} = {H_3}{s_{a,b}^{-3/2}}$,  using an approximate Fourier expansion of ${H_3(\theta_s)}$
gives
\begin{equation}\FC{\tau_3}_n \approx
\sum_{k=-p}^{p}
\FC{H_3}_k
\FC{s_{a,b}^{-3/2}}_{n-k},
\label{eq:tau3-standard}
\end{equation}
where $p$ again is determined only by a rule similar to \eqref{eq:crt:M0}.
Define $B_m(a,b):= \FC{s_{a,b}^{-3/2}}_m, m\in\mathbb{Z}$.  We
have 
\begin{prop}
\label{lem:Bm}
$B_m$ satisfy
\begin{equation}B_m(a,b)= e^{-\ii ma}B_{|m|}(0,b),
\quad
B_{-m}=\overline{B_m},\quad m\in\mathbb Z.
\label{eq:B3-recovery}
\end{equation}
For $m\geq 1$, $B_m$ satisfy the forward recurrence
\begin{align}B_m(a,b)
&=e^{-\bi a} \cosh b\,B_{m-1}(a,b) -\frac{2m-1}{2}e^{-\bi m a + b}A_{m-1}(b),
\label{eq:Bm-recurrence}
\end{align}
where $A_m$ is defined in \eqref{eq:Am}. In particular, the initial value is
\begin{equation}B_0(a,b)= \frac{4E(e^{-2b})}{\pi(1-e^{-2b})^2}
-\frac{A_0(b)}{1-e^{-2b}},
\label{eq:B0}
\end{equation}
where $E(x)=\int_0^{\pi/2}(1-x\sin^2\phi)^{1/2}\,d\phi$
denotes the complete elliptic integral of the second kind.
\begin{proof}
By the definition of $B_m$, one readily verifies \eqref{eq:B3-recovery}.
For $m\geq 1$,
\begin{align*}B_m(0,b)
&=\frac{1}{2\pi}\int_{-\pi}^{\pi}
s_{0,b}^{-3/2}e^{-\ii m\theta_s}\,d\theta_s\nonumber\\
&=\frac{1}{2\pi}\int_{-\pi}^{\pi}
s_{0,b}^{-3/2}e^{-\ii(m-1)\theta_s}
\cos\theta_s\,d\theta_s
+e^{b} \frac{\ii}{2\pi}\int_{-\pi}^{\pi}
\partial_{\theta_s} s_{0,b}^{-1/2}e^{-\ii(m-1)\theta_s}
\,d\theta_s,
\end{align*}
where we have used the identity $\partial_{\theta_s}s_{0,b}^{-1/2}
=-e^{-b}\sin\theta_s\,s_{0,b}^{-3/2}$.
Using $\cos\theta_s=\cosh b-\frac12e^b s_{0,b}(\theta_s)$ for the first integral and integration by parts for the second integral,
\begin{align*}B_m(0,b) &= \cosh b\,B_{m-1}(0,b) -\frac{1}{2}e^bA_{m-1}(b) - (m-1)e^{b} A_{m-1}(b)\\
&=\cosh b\,B_{m-1}(0,b) -\frac{2m-1}{2}e^bA_{m-1}(b).
\end{align*}
Multiplying both sides by $e^{-\bi m a}$ gives
\eqref{eq:Bm-recurrence}. For $m=0$,  the identity
\[(1-e^{-2b})s_{0,b}^{-3/2}=s_{0,b}^{-1/2}-2\partial_b s_{0,b}^{-1/2}
\]
gives
\[(1-e^{-2b})B_0(0,b)=A_0(b)-2A_0'(b) = A_0(b) - 2\left[A_0(b)-\frac{2E(e^{-2b})}{\pi(1-e^{-2b})}\right],
\]
which implies \eqref{eq:B0} by $B_0(a,b)=B_0(0,b)$. Note the last equality
follows from standard identities for complete elliptic integrals; see
\cite[Secs.~19.4 and 19.8]{DLMF}.
\end{proof}
\end{prop}
According to Proposition~\ref{lem:Bm}, all $B_m(a,b)$ are evaluated in ${\cal O}(M)$
operations.  Consequently, $\FC{\tau_3}_n$ can be evaluated in ${\cal O}(M\log
M)$ operations similarly to $\FC{\tau_1}_n$.

For sufficiently small $b$, the stronger singular factor
$s_{a,b}^{-3/2}$ amplifies the round-off errors in \eqref{eq:tau3-standard}. To
tackle this issue,  following the translated singularity swapping approach in
\cite{krantz2026stabilizing}, we instead use
\begin{equation}H_3(\theta_s)
\approx
H_3(a)+H_3'(a)\sin(\theta_s-a)
+
\sin^2\frac{\theta_s-a}{2}
\sum_{k=-(p-1)}^{p-1}d_k e^{\ii k\theta_s}.
\label{eq:H3-modified}
\end{equation}
Let
\begin{equation}\widetilde B_m(a,b) = \FC{ \sin^2\frac{\theta_s-a}{2}\, s_{a,b}^{-3/2} }_m,
\qquad m\in\mathbb Z,
\label{eq:B3-modified}
\end{equation}
and then
\begin{align}    \FC{\tau_3}_n
    &\approx H_3(a)B_n(a,b)
    +
    \frac{H_3'(a)}{2\ii}
    \left[
        e^{-\ii a}B_{n-1}(a,b)
        -
        e^{\ii a}B_{n+1}(a,b)
    \right]
    \nonumber\\
    &+
    \sum_{k=-(p-1)}^{p-1}
    d_k\widetilde B_{n-k}(a,b).
    \label{eq:tau3-modified}
\end{align}
Consequently, performing FFT convolutions twice, one for $\widetilde B_m$, gives
$\FC{\tau_3}_n$.

In summary, \eqref{eq:tau1} is adopted for $\tau_1$, while
\eqref{eq:tau3-standard} or \eqref{eq:tau3-modified} is adopted for $\tau_3$,
depending on the value of $b$; throughout this paper, we consider $b<0.01$ to be
sufficiently small. For each fixed $t,\theta,t_s$, the computational cost for
obtaining $\left\{\FC{\tau_j}_l\right\}_{l=0}^{M+M_0}, j=1,3,$ is ${\cal O}(M\log M)$.

\begin{remark}
The small quantity $\tau_0=O(\rho^2)$ should be evaluated carefully to
avoid significant cancellation errors, which can occur when
evaluating $\FC{\tau_3}_l$. To tackle this issue, we adopt the
technique developed in \cite{luluqia18}. Let
${\bm u}=(t,\theta)$, ${\bm v}=(t_s, \theta_s)$,
${\bm \delta}=[\delta t, \delta \theta]:={\bm u}-{\bm v}$, and
${\bm \gamma}(\lambda)={\bm v}+\lambda{\bm \delta}$. Since
$\bnu({\bm v})$ is orthogonal to the tangent plane, we write
\begin{align*} 
&\tau_0({\bm v})
    =
    \int_0^1(\lambda-1)\,
    \bnu({\bm v})\cdot
    \left[\br_{tt}({\bm \gamma}(\lambda))(\delta t)^2
+2\br_{t\theta}({\bm \gamma}(\lambda))
\delta t\,\delta\theta +
\br_{\theta\theta}({\bm \gamma}(\lambda))
(\delta\theta)^2\right] \,d\lambda,\\
    &\br({\bm v})-\br({\bm u})
    = -\int_0^1 D\br({\bm \gamma}(\lambda))[{\bm u}-{\bm v}]\,d\lambda.
    \label{eq:r-stable}
\end{align*}
These integrals can be accurately evaluated by $16$-point Gauss-Legendre rules.
\end{remark}

Similar to the single-layer operator ${\cal S}$, by applying the same quadrature
rule as discussed in Section~\ref{sec31}, the double-layer operator ${\cal K}$ can also
be approximated as an $N\times N$ matrix ${\bm K}$ by collocating both
$(t,\theta)$ and $(t_s,\theta_s)$ at $(t_i^j, \theta_l), i=1\cdots N_p,
j=1,\cdots 16, l=0,\cdots, 2M-1$. Therefore, \eqref{eq:bie:Kie} can be respectively approximated by
\[({\bm K} \mp {\bm I}){\bm \psi} = {\bm g}.
\]
One can similarly discretize the adjoint double-layer operator ${\cal K}'$ in
\eqref{eq:bie:Kpie}. We omit the details.

As in the last paragraph of Section~\ref{sec3},  we see that the correction
stages for discretizing ${\cal K}$ and ${\cal K}'$ have the same time complexity
${\cal O}\bigl(M^2\log M\bigr)$.
Figures~\ref{pictimeK}
\begin{figure}[htb]
    \centering
    \includegraphics[width=0.45\textwidth]{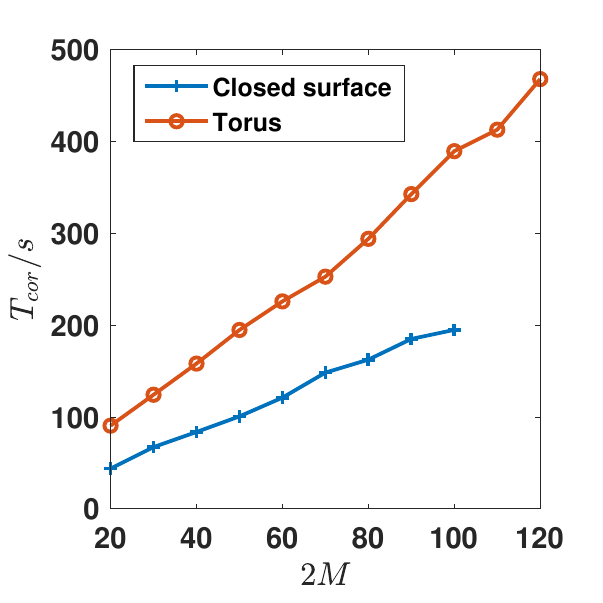}
    \includegraphics[width=0.45\textwidth]{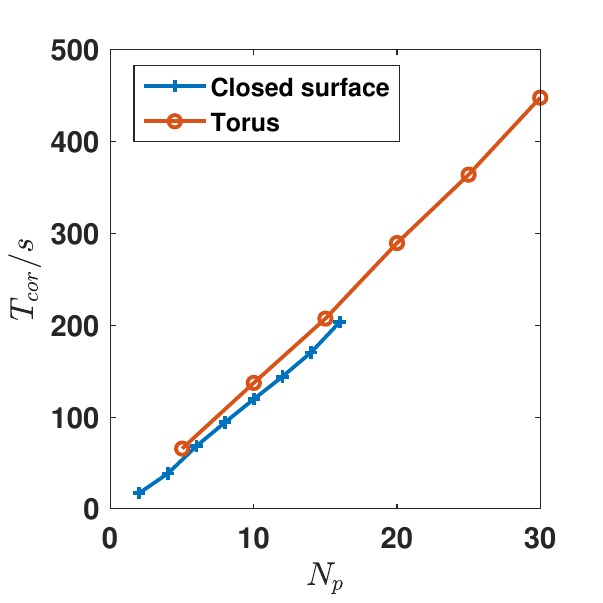}
    \caption{$T_{cor}$ against $2M$ and $N_p$ for ${\cal K}$: (a) $N_p$ is $30$ for 'o's and $16$ for '$+$'s; (b) $2M$ is $120$ for 'o's and $100$ for '$+$'s.}
    \label{pictimeK}
\end{figure}
and~\ref{pictimeKp}
\begin{figure}[htb]
    \centering
    \includegraphics[width=0.45\textwidth]{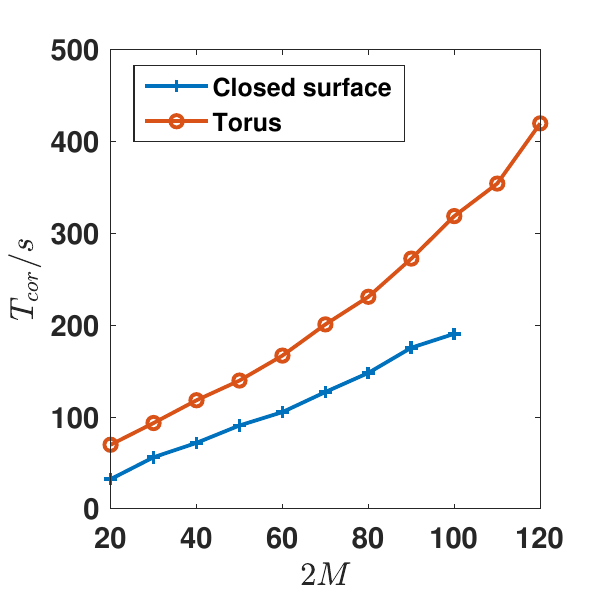}
    \includegraphics[width=0.45\textwidth]{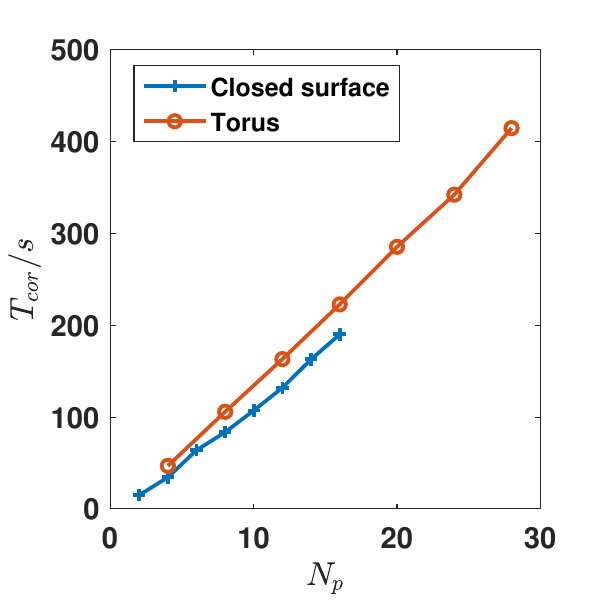}
    \caption{$T_{cor}$ against $2M$ and $N_p$ for ${\cal K}'$: (a) $N_p$ is $30$ for 'o's and $16$ for '$+$'s; (b) $2M$ is $120$ for 'o's and $100$ for '$+$'s.}
    \label{pictimeKp}
\end{figure}
show the running time $T_{cor}$  in the correction stage for the two surfaces mentioned for Figure~\ref{pictimeS} for several values of $2M$ and $N_p$. The increasing rates of $T_{cor}$ with $2M$ and $N_p$ in the plots still are consistent with the expected time complexity.

\section{Numerical examples}\label{sec5}

In this section, we carry out several experiments to check the accuracy of the
proposed discretization schemes for ${\cal S}$, ${\cal K}'$, and ${\cal K}$.
Without loss of generality, we study the interior problems (IDP) and (INP) only.
Let $\br_*$ be a specified source point located in
$\Omega_e$ and let
  \[
    g(\br) = \left\{
      \begin{array}{ll}
      G(\br;\br_*) & {\rm (IDP)},\\
      \partial_{\nu(\br)}G(\br;\br_*) & {\rm (INP)},\\
      \end{array}
  \right.
  \]
for $\br\in\Gamma$. Clearly, $u_{\rm exa}(\br) = G(\br;\br_*)$ is the exact
solution for the two problems (IDP) and (INP) and shall be used to compute the
numerical error
\[  E_{j} = \frac{||{\bm u}^j_{\rm num} - {\bm u}_{\rm
      exa}||_{2}^{\Omega_T}}{\#({\bm u}_{\rm exa})},
\]
for the corresponding integral operator $j={\cal S}, {\cal K}, {\cal K}'$. Here $\Omega_T$ is the testing region where the corresponding numerical solutions $u^j_{\rm num}$
are evaluated; ${\bm u}^j_{\rm num}$ and ${\bm u}_{\rm exa}$ are the vectors of
$u^j_{\rm num}$ and the exact solution $u_{\rm exa}$ at the grid points in
$\Omega_T$, respectively; $\#$ indicates the size of the vector, and $||\cdot||_{2}^{\Omega_T}$ denotes the
$l^{2}$-norm of the vector.

In all examples, we choose the source point ${\bm r}_*= [3,4,10.5]^{T}\in \Omega_e$ and $\Omega_T$ is chosen to be a ball $|{\bm r}-{\bm r}_0|< \delta$ with its center ${\bm r}_0$ and radius $\delta$ given by:

\begin{description}
\item[Example 1.] $\bm r_0=[0,-0.2,0.1]^T$ and $\delta=0.2$;
\item[Example 2.] $\bm r_0=[0.2,-0.2,0.1]^T$ and $\delta=0.2$;
\item[Example 3.] $\bm r_0=[2.75,-0.2,0.1]^T$ and $\delta=0.2$;
\item[Example 4.] $\bm r_0=[4.75,-0.2,0.1]^T$ and $\delta=0.1$.
\end{description}

\noindent {\bf Example 1: An ellipsoid.} The surface $\Gamma$ is characterized by the following equation:
\[\frac{x^2}{1^2} + \frac{y^2}{2^2} + \frac{z^2}{3^2} = 1.
\]

\begin{table}[ht]
\centering
\begin{tabular}{ c|c|c|c|c|c|c}
\hline
 $k$ & $(N_p,2M)$  & $E_{\cal S}$ & $(N_p,2M)$  & $E_{\cal K}$ & $(N_p,2M)$  & $E_{{\cal K}'}$ \\\hline
1 & $(4,30)$ & 6.4E-11 & $(4,40)$ & 7.6E-12 & $(4,40)$ & 1.3E-11\\\hline
2 & $(4,30)$ & 1.2E-10 & $(4,40)$ & 1.6E-11 & $(4,40)$ & 9.3E-12 \\\hline
4 & $(4,30)$ & 8.4E-10 & $(4,40)$ & 1.2E-10 & $(4,40)$ & 2.6E-10 \\\hline
8 & $(8,60)$ & 9.7E-11 & $(8,60)$ & 2.3E-11& $(8,60)$ & 2.2E-10 \\\hline
16 & $(16,100)$ & 9.6E-10 & $(16,100)$& 7.2E-11 & $(16,100)$ & 8.0E-12\\\hline
\end{tabular}
\caption{Example 1: Error $E_j, j={\cal S}, {\cal K}$ and ${\cal K}'$ for different values of $k$.}
\label{tab:ex1}
\end{table}

Table~\ref{tab:ex1} shows the accuracy of ${\cal S}$, ${\cal K}$ and ${\cal K}'$ for specific values of $2M$ and $N_p$ such that the errors are less than $10^{-9}$. Figure~\ref{picconv1}
\begin{figure}[htb]
    \centering
    \includegraphics[width=0.32\textwidth]{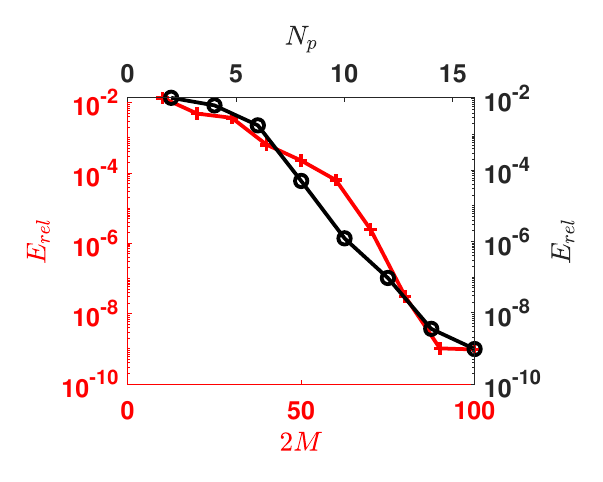}
    \includegraphics[width=0.32\textwidth]{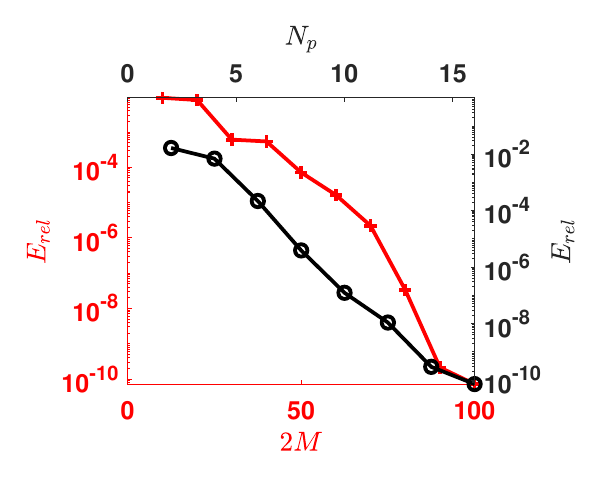}
    \includegraphics[width=0.32\textwidth]{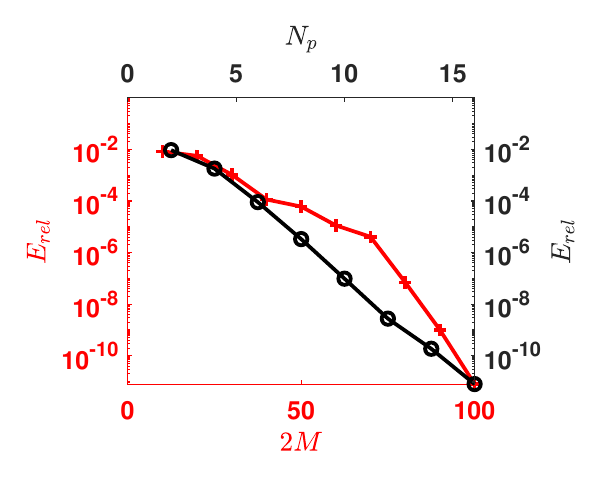}
    \caption{Example 1: Convergence curves of ${\cal S}$, ${\cal K}$ and ${\cal K}'$ for several values of $2M$ and $N_p$ for $k=16$: $N_p$ is $16$ for all '$+$'s and $2M=100$ for all 'o's.}
    \label{picconv1}
\end{figure}
plots the discretization errors of ${\cal S}$, ${\cal K}$ and ${\cal K}'$ for different values of $2M$ and $N_p$ for $k=16$. With the vertical axis logarithmically scaled, it can be seen that the errors decay exponentially as $M$ and $N_p$ increase. Figure~\ref{picfield1}
\begin{figure}[htb]
    \centering
    \includegraphics[width=0.32\textwidth]{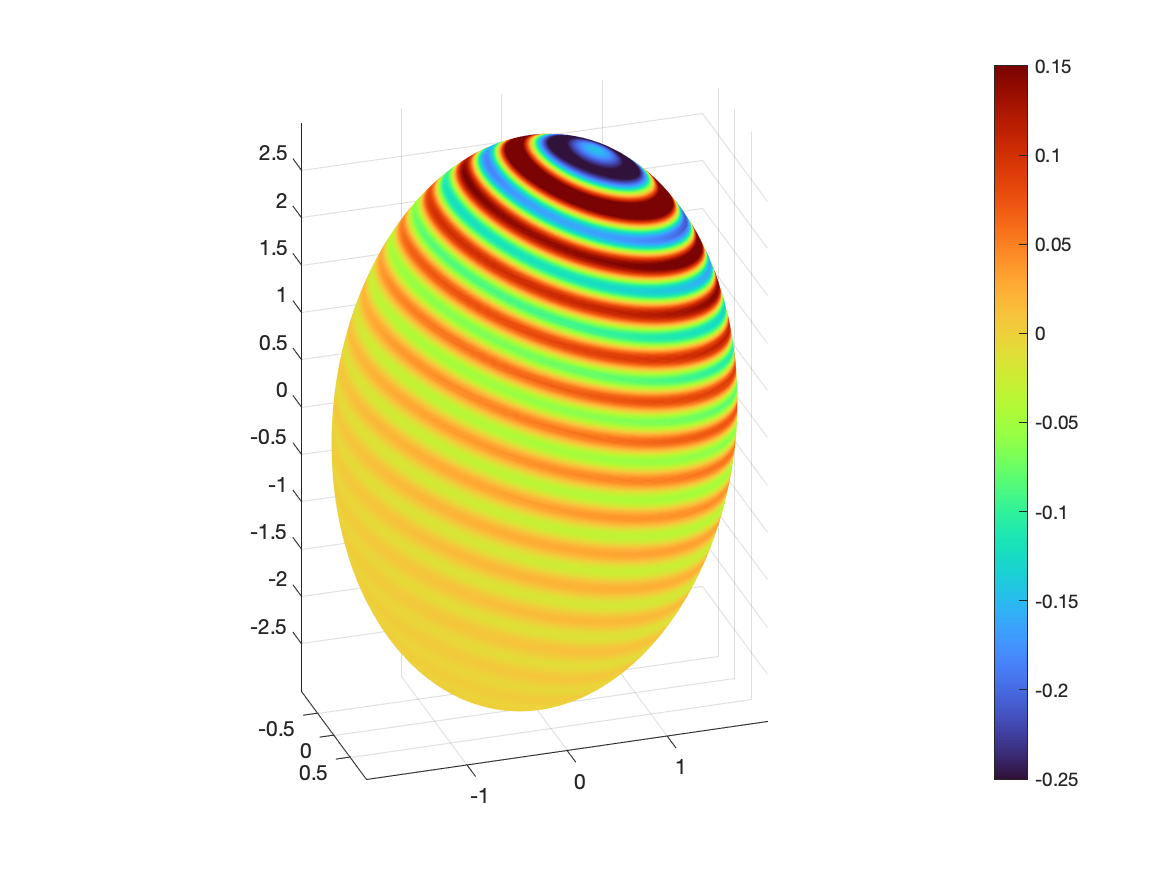}
    \includegraphics[width=0.32\textwidth]{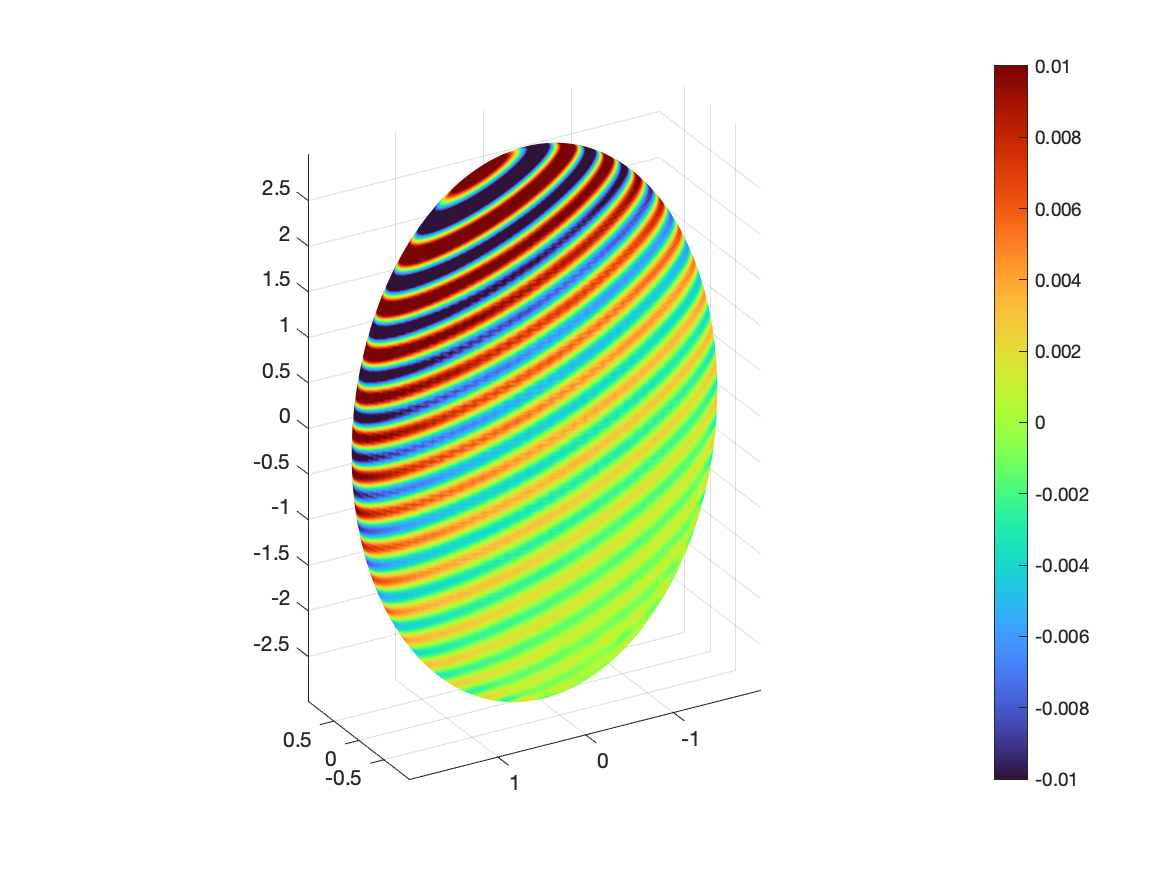}
    \includegraphics[width=0.32\textwidth]{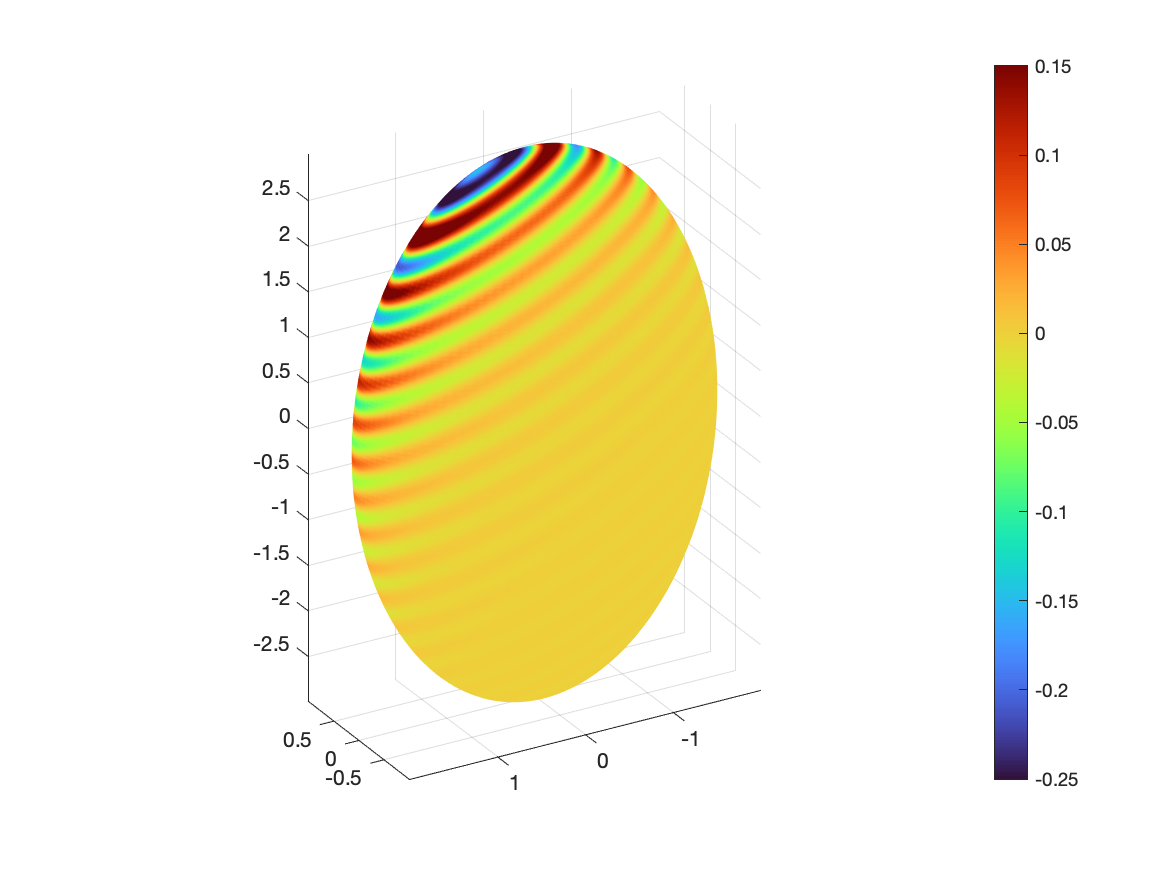}
    \caption{Example 1: Real parts of $\phi$ and $\psi$ on $\Gamma$ for $k=16$ for ${\cal S}$ (left), ${\cal K}$ (middle) and ${\cal K}'$ (right); $N_p$ and $2M$ are chosen according to the last row in Table~\ref{tab:ex1}.}
    \label{picfield1}
\end{figure}
plots the real parts of the unknown density functions $\phi$ in equations~\eqref{eq:bie:S} and~\eqref{eq:bie:Kpie} and $\psi$ in equation~\eqref{eq:bie:Kie} on $\Gamma$.

\noindent {\bf Example 2: A bean-shaped surface}.
The equation of the surface is:
\[\frac{x^2}{0.64(1-0.1\cos(\pi z))} + \frac{(y+0.3\cos(\pi z))^2}{0.64(1-0.4\cos(\pi z))} + z^2 = 1.
\]

Table~\ref{tab:ex2} shows the accuracy of ${\cal S}$, ${\cal K}$ and ${\cal K}'$ for specific values of $2M$ and $N_p$ such that the errors are less than $10^{-9}$.
\begin{table}[ht]
\centering
\begin{tabular}{ c|c|c|c|c|c|c}
\hline
 $k$ & $(N_p,2M)$  & $E_{\cal S}$ & $(N_p,2M)$  & $E_{\cal K}$ & $(N_p,2M)$  & $E_{{\cal K}'}$ \\\hline
1 & $(4,30)$ & 1.0E-12 & $(4,40)$ & 9.1E-13 & $(4,40)$ & 8.3E-13\\\hline
2 & $(4,30)$ & 1.6E-12 & $(4,40)$ & 1.3E-12 & $(4,40)$ & 2.1E-12 \\\hline
4 & $(4,30)$ & 1.1E-11 & $(4,40)$ & 9.5E-12 & $(4,40)$ & 2.2E-12 \\\hline
8 & $(8,60)$ & 5.2E-12 & $(8,60)$ & 4.8E-13 & $(8,60)$ & 1.0E-11 \\\hline
16 & $(16,100)$ & 4.7E-12 & $(16,100)$ & 3.8E-13 & $(16,100)$ & 1.3E-11\\\hline
\end{tabular}
\caption{Example 2: Error $E_j, j={\cal S}, {\cal K}$ and ${\cal K}'$ for different values of $k$.}
\label{tab:ex2}
\end{table}
Figure~\ref{picconv2}
\begin{figure}[htb]
    \centering
    \includegraphics[width=0.32\textwidth]{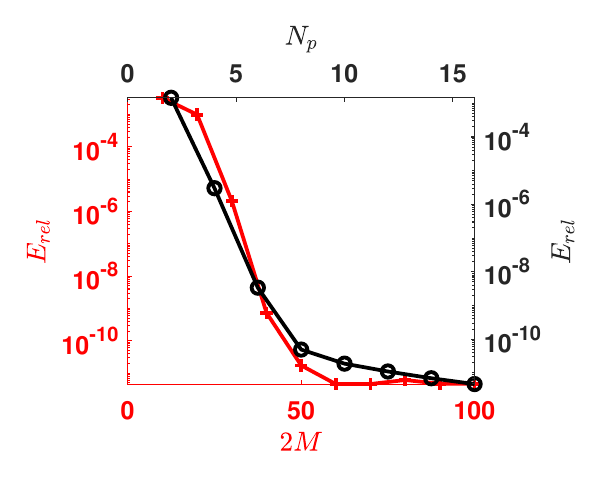}
    \includegraphics[width=0.32\textwidth]{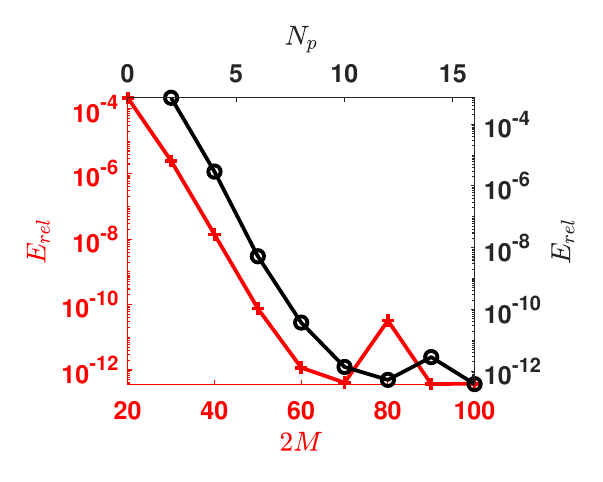}
    \includegraphics[width=0.32\textwidth]{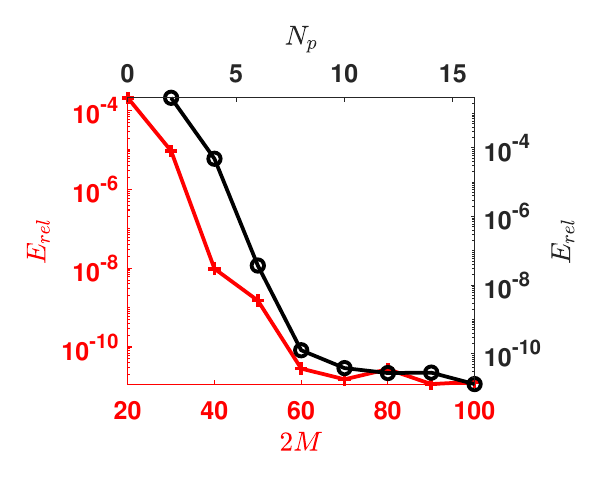}
    \caption{Example 2: Convergence curves of ${\cal S}$, ${\cal K}$ and ${\cal K}'$ for several values of $2M$ and $N_p$ for $k=16$: $N_p$ is $16$ for all '$+$'s and $2M=100$ for all 'o's.}
    \label{picconv2}
\end{figure}
plots the discretization errors of ${\cal S}$, ${\cal K}$ and ${\cal K}'$ for different values of $2M$ and $N_p$ for $k=16$. Figure~\ref{picfield2}
\begin{figure}[htb]
    \centering
    \includegraphics[width=0.32\textwidth]{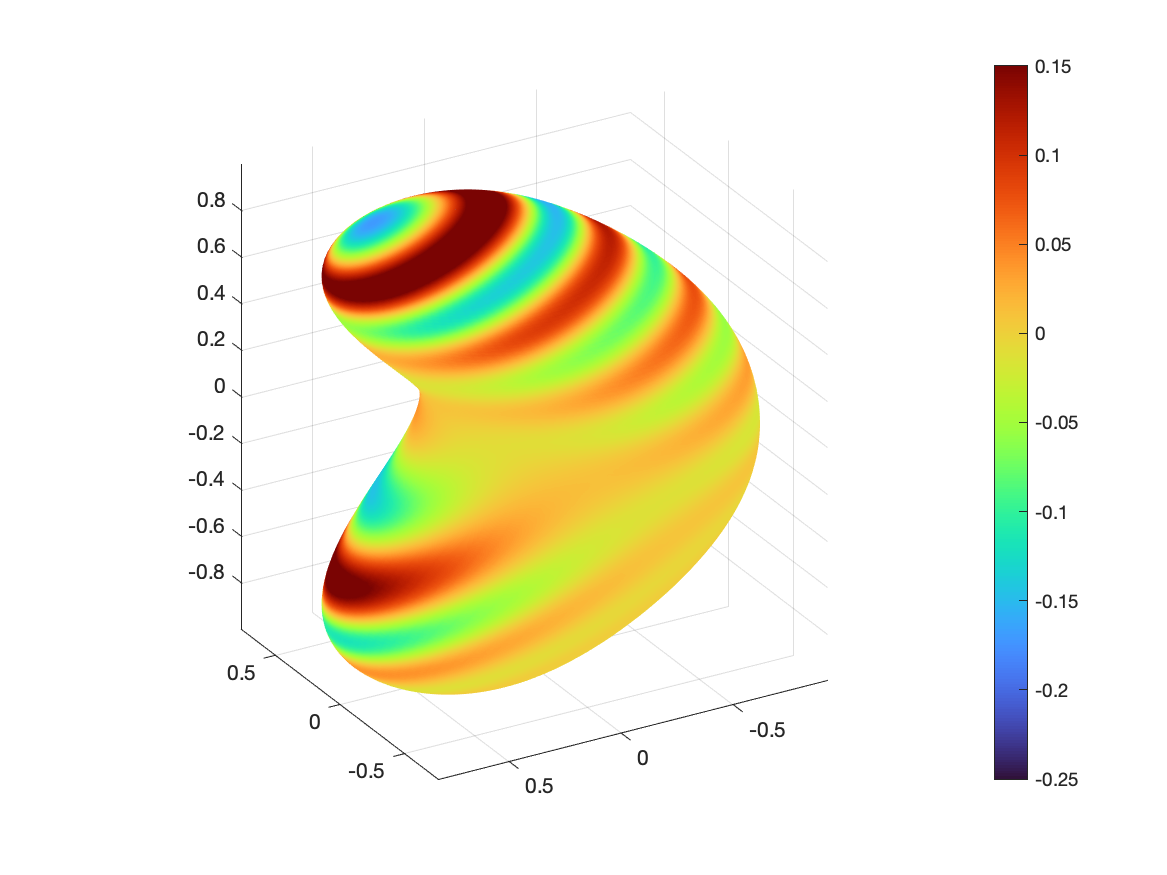}
    \includegraphics[width=0.32\textwidth]{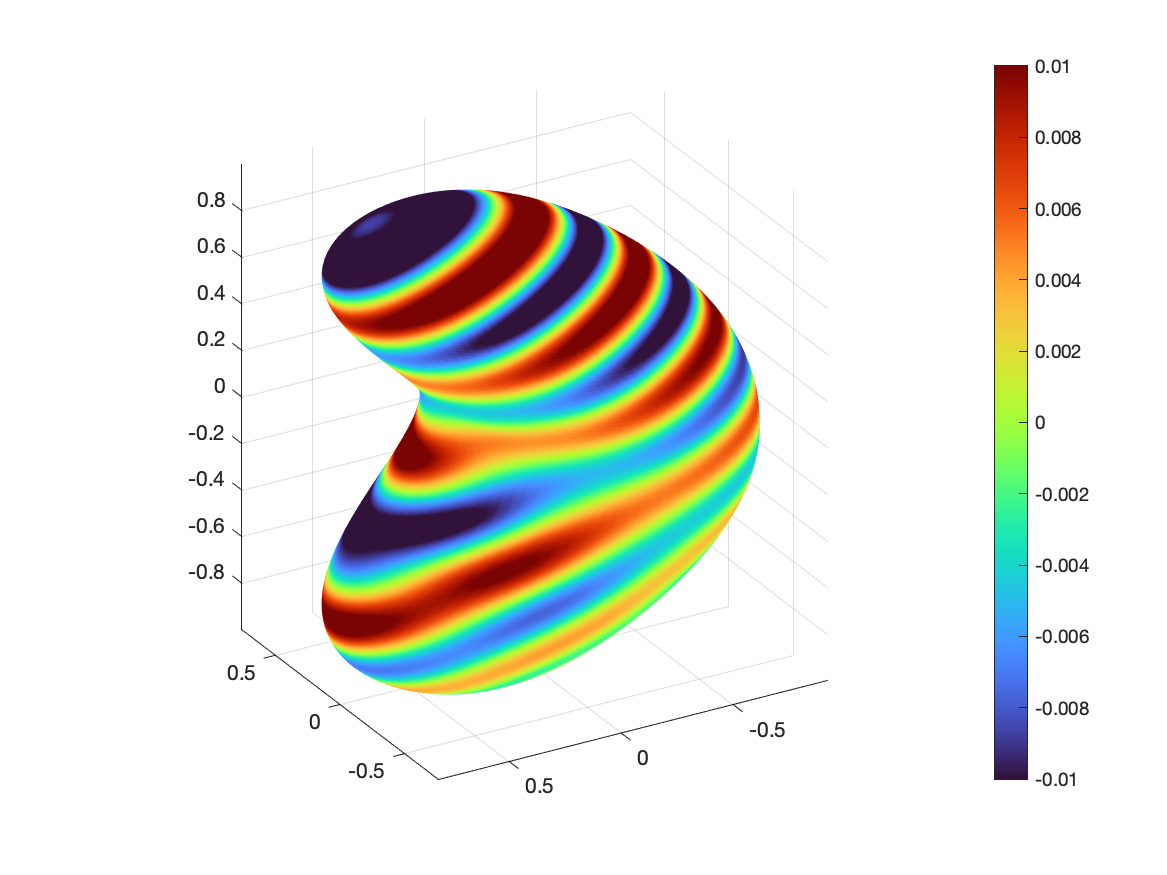}
    \includegraphics[width=0.32\textwidth]{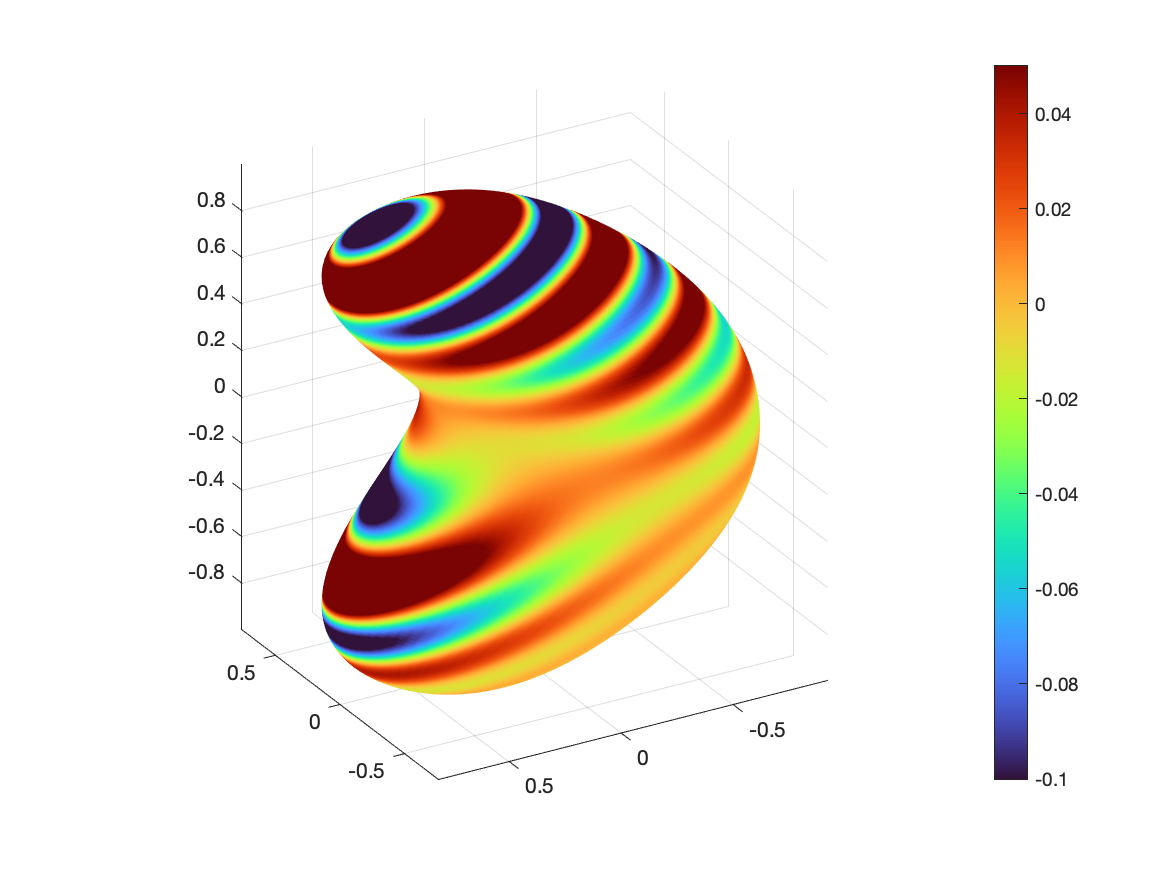}
    \caption{Example 2: Real parts of $\phi$ and $\psi$ on $\Gamma$ for $k=16$ for ${\cal S}$ (left), ${\cal K}$ (middle) and ${\cal K}'$ (right); $N_p$ and $2M$ are chosen according to the last row in Table~\ref{tab:ex2}.}
    \label{picfield2}
\end{figure}
plots the real parts of the unknown density functions $\phi$ in equations~\eqref{eq:bie:S} and~\eqref{eq:bie:Kpie} and $\psi$ in equation~\eqref{eq:bie:Kie} on $\Gamma$.

\noindent {\bf Example 3: A wiggly torus.}
The surface is parameterized,  for $(t,\theta)\in[0,2\pi]^2$, by
\begin{align*}x(t,\theta) &= 1.2(2+\cos(\theta) + 0.25\cos(5t))\cos(t),\\
y(t,\theta) &= (2+\cos(\theta) + 0.25\cos(5t))\sin(t),\\
z(t,\theta) &= 1.7\sin(\theta),
\end{align*}

Table~\ref{tab:ex3} shows the accuracy of ${\cal S}$, ${\cal K}$ and ${\cal K}'$ for specific values of $2M$ and $N_p$ such that the errors are less than $10^{-9}$.
\begin{table}[ht]
\centering
\begin{tabular}{ c|c|c|c|c|c|c}
\hline
 $k$ & $(N_p,2M)$  & $E_{\cal S}$ & $(N_p,2M)$  & $E_{\cal K}$ & $(N_p,2M)$  & $E_{{\cal K}'}$ \\\hline
1 & $(5,30)$ & 8.6E-11 & $(9,30)$ & 3.2E-11 & $(5,30)$ & 1.0E-10\\\hline
2 & $(7,30)$ & 2.5E-10 & $(10,40)$ & 2.4E-11 & $(9,50)$ & 7.8E-11 \\\hline
4 & $(9,30)$ & 6.8E-10 & $(10,40)$ & 1.6E-10 & $(9,50)$ & 3.4E-10 \\\hline
8 & $(14,60)$ & 3.6E-10 & $(14,80)$ & 3.5E-10 & $(14,60)$ & 3.0E-10 \\\hline
16 & $(28,120)$ & 1.8E-10 & $(30,120)$ & 3.3E-10 & $(28,120)$ & 7.0E-11\\\hline
\end{tabular}
\caption{Example 3: Error $E_j, j={\cal S}, {\cal K}$ and ${\cal K}'$ for different values of $k$.}
\label{tab:ex3}
\end{table}
Figure~\ref{picconv3}
\begin{figure}[htb]
    \centering
    \includegraphics[width=0.32\textwidth]{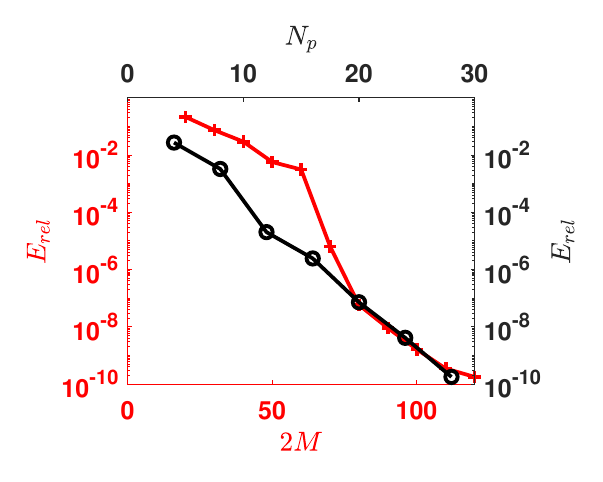}
    \includegraphics[width=0.32\textwidth]{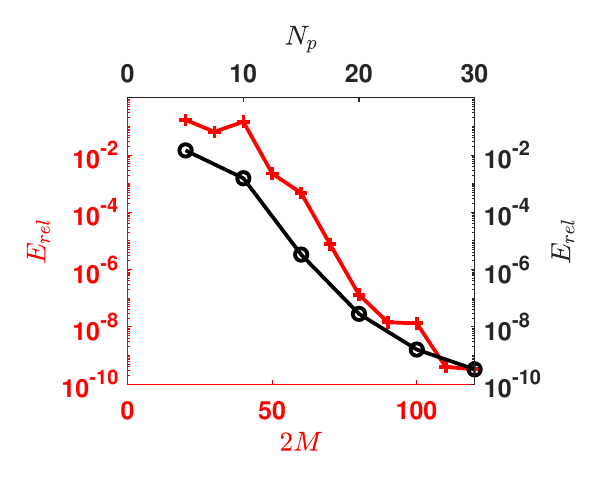}
    \includegraphics[width=0.32\textwidth]{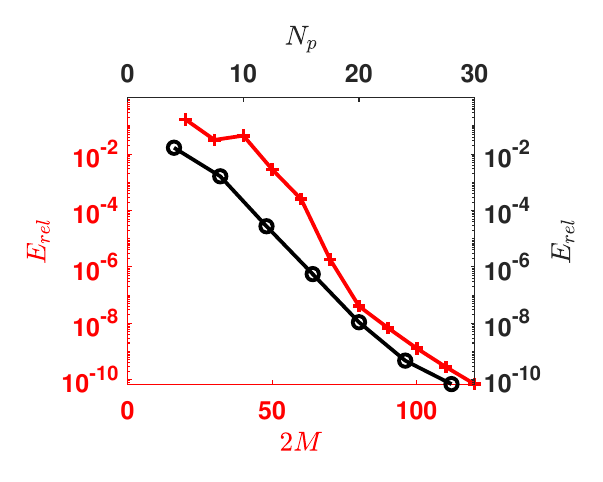}
    \caption{Example 3: Convergence curves of ${\cal S}$, ${\cal K}$ and ${\cal K}'$ for several values of $2M$ and $N_p$ for $k=16$: $N_p$ is $28$ for '$+$'s in the left, $30$ for '$+$'s in the middle, and $28$ for '$+$'s in the right and $2M$ is $100$ for 'o's in the left, $120$ for 'o's in the middle, and $120$ for 'o's in the right.}
    \label{picconv3}
\end{figure}
plots the discretization errors of ${\cal S}$, ${\cal K}$ and ${\cal K}'$ for different values of $2M$ and $N_p$ for $k=16$. Figure~\ref{picfield3}
\begin{figure}[htb]
    \centering
    \includegraphics[width=0.32\textwidth]{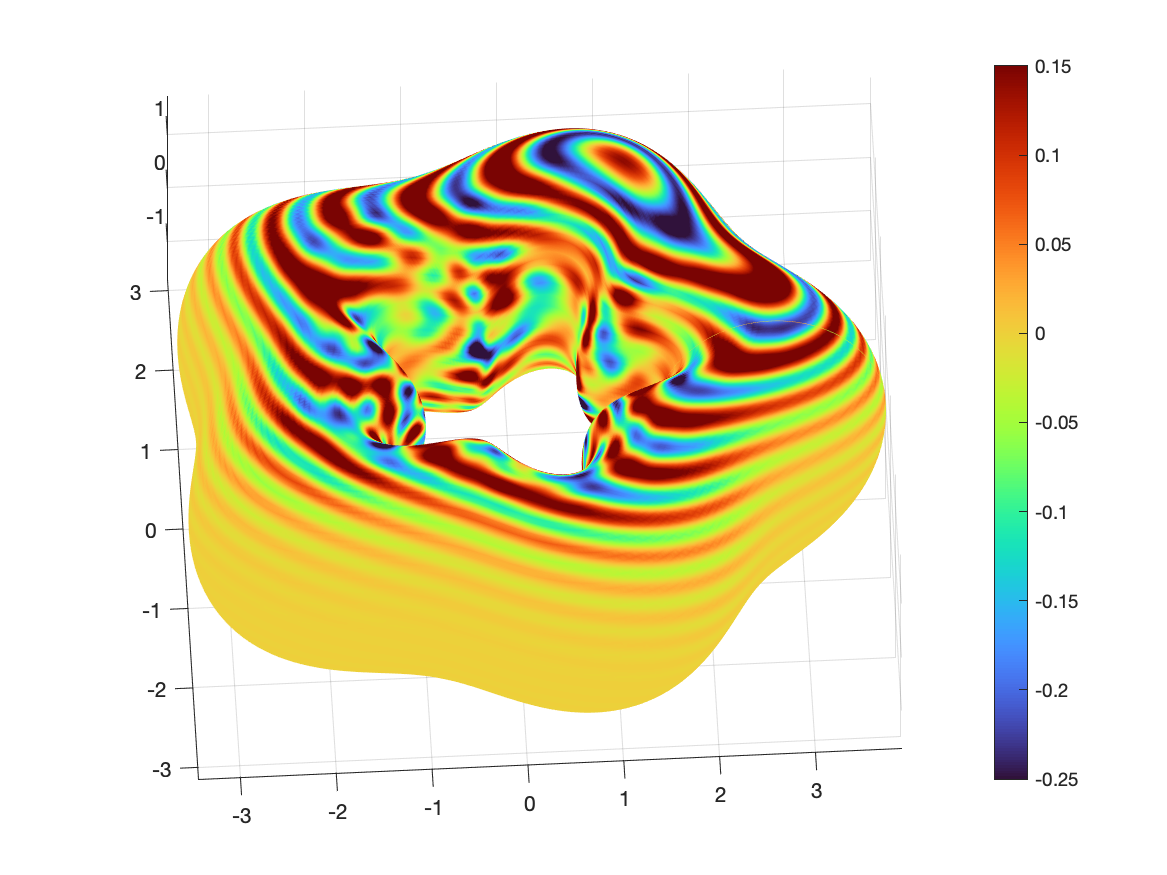}
    \includegraphics[width=0.32\textwidth]{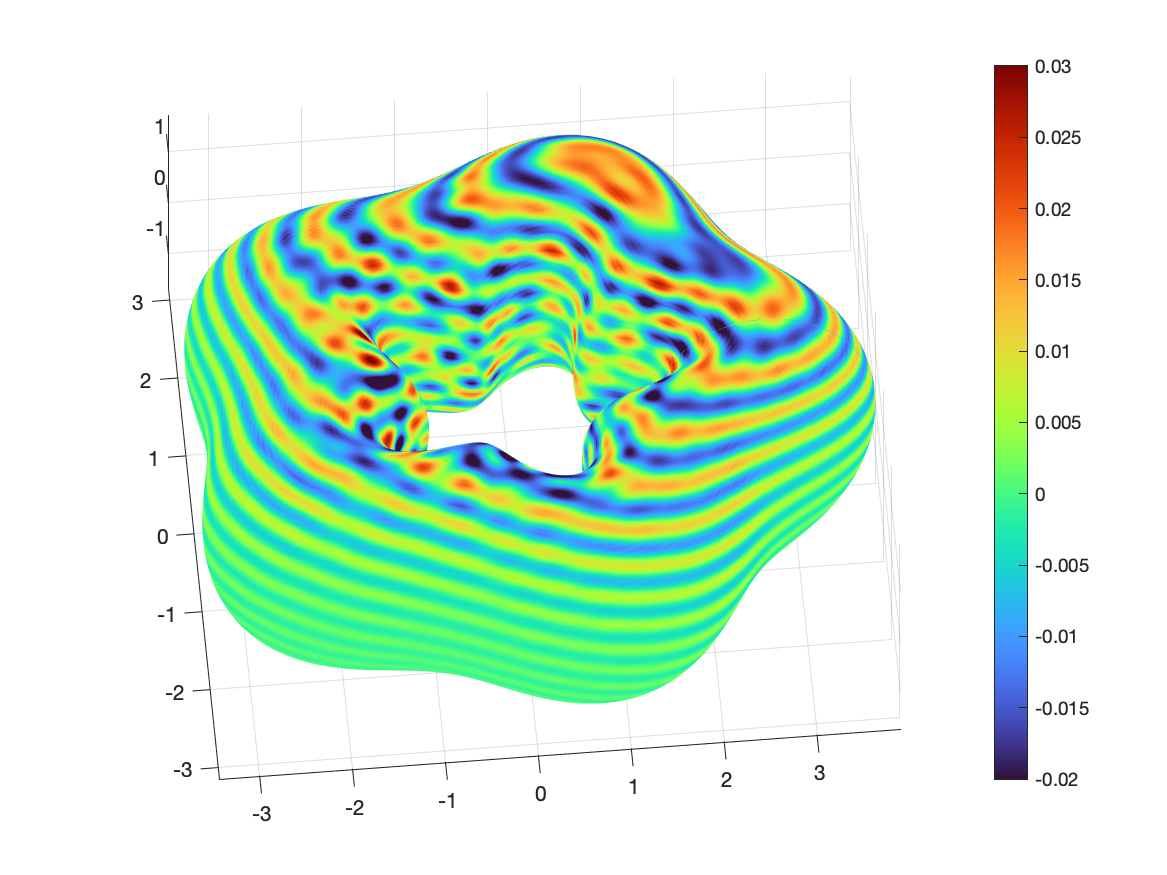}
    \includegraphics[width=0.32\textwidth]{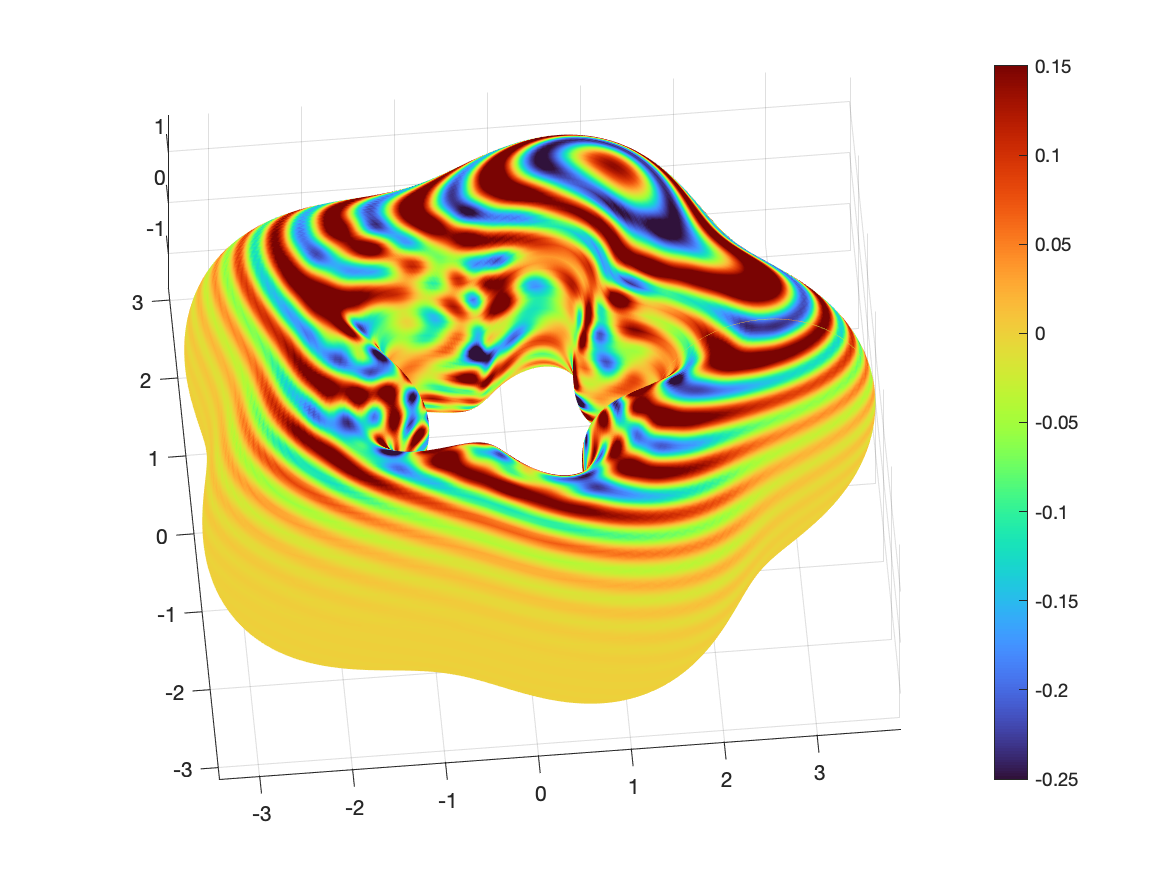}
    \caption{Example 3: Real parts of $\phi$ and $\psi$ on $\Gamma$ for $k=16$ for ${\cal S}$ (left), ${\cal K}$ (middle) and ${\cal K}'$ (right); $N_p$ and $2M$ are chosen according to the last row in Table~\ref{tab:ex3}.}
    \label{picfield3}
\end{figure}
plots the real parts of the unknown density functions $\phi$ in equations~\eqref{eq:bie:S} and~\eqref{eq:bie:Kpie} and $\psi$ in equation~\eqref{eq:bie:Kie} on $\Gamma$.

\noindent {\bf Example 4: A twisted torus.} This last challenging example has been previously studied in \cite{greengard2021fast}.
The equation of the surface $\Gamma$ is given by
\begin{align*}  x(t,\theta) &= \cos(t)(0.17\cos(2\theta-t) + 0.11\cos(2\theta) + \cos(\theta) + 4.5\\
  &- 0.25\cos(\theta) + 0.01\cos(\theta+t) - 0.45\cos(-\theta+t)),\\
  y(t,\theta) &= \sin(t)(0.17\cos(2\theta-t) + 0.11\cos(2\theta) + \cos(\theta) + 4.5\\
  &- 0.25\cos(\theta) + 0.01\cos(\theta+t) - 0.45\cos(-\theta+t)),\\
  z(t,\theta) &= (0.17\sin(2\theta-t) + 0.11\sin(2\theta) + \sin(\theta) - 0.25\sin(\theta)\\
  &+ 0.01\sin(\theta+t) - 0.45\sin(-\theta+t)).
\end{align*}
Table~\ref{tab:ex4} shows the accuracy of ${\cal S}$, ${\cal K}$ and ${\cal K}'$ for specific values of $2M$ and $N_p$ such that the errors are less than $10^{-9}$.
\begin{table}[ht]
\centering
\begin{tabular}{ c|c|c|c|c|c|c}
\hline
 $k$ & $(N_p,2M)$  & $E_{\cal S}$ & $(N_p,2M)$  & $E_{\cal K}$ & $(N_p,2M)$  & $E_{{\cal K}'}$ \\\hline
1 & $(12,100)$ & 3.9E-10 & $(22,120)$ & 5.0E-10 & $(14,120)$ & 9.4E-10\\\hline
2 & $(12,100)$ & 5.5E-10 & $(22,120)$ & 3.2E-10 & $(16,120)$ & 4.4E-10 \\\hline
4 & $(12,100)$ & 9.1E-10 & $(22,120)$ & 3.7E-10 & $(22,120)$ & 2.0E-10 \\\hline
8 & $(14,160)$ & 8.3E-10 & $(28,160)$ & 1.8E-10 & $(22,120)$ & 4.1E-10 \\\hline
\end{tabular}
\caption{Example 4: Error $E_j, j={\cal S}, {\cal K}$ and ${\cal K}'$ for different values of $k$.}
\label{tab:ex4}
\end{table}
Figure~\ref{picfield4}
\begin{figure}[htb]
    \centering
    \includegraphics[width=0.32\textwidth]{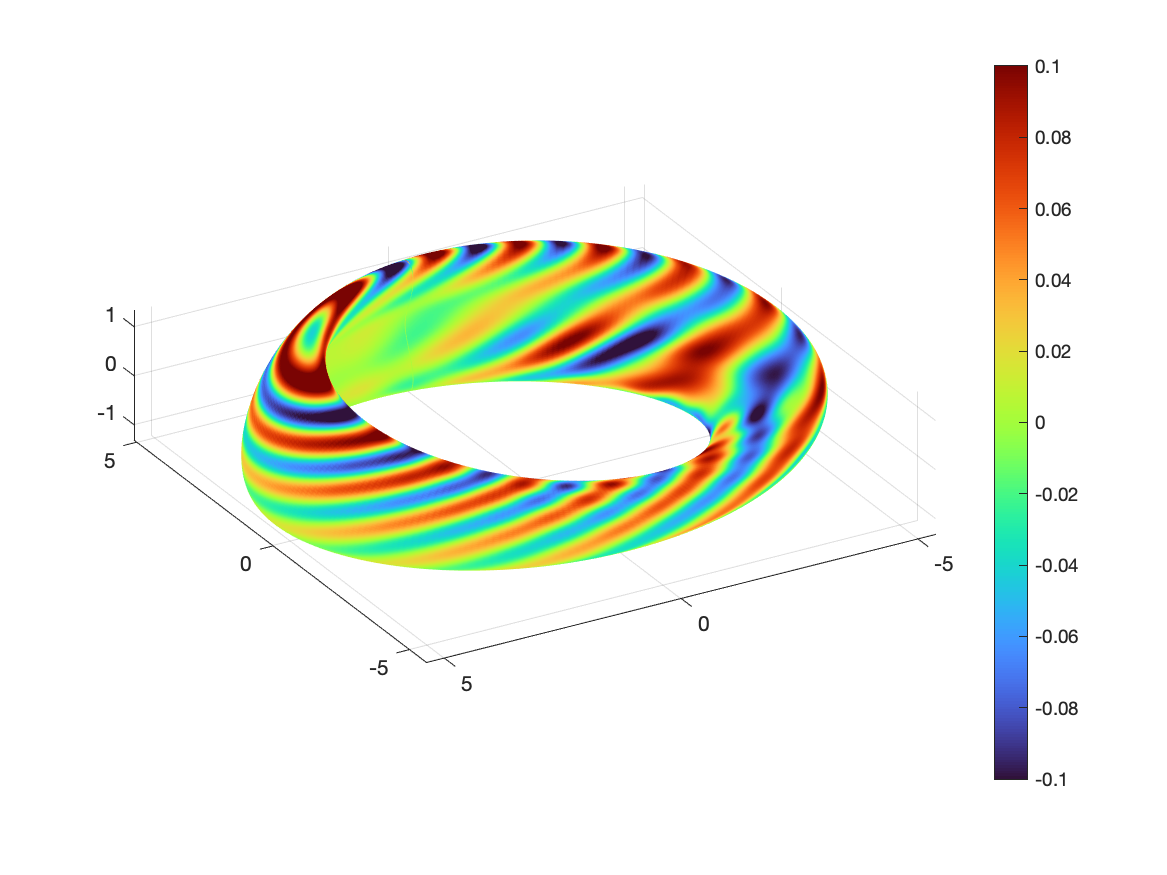}
    \includegraphics[width=0.32\textwidth]{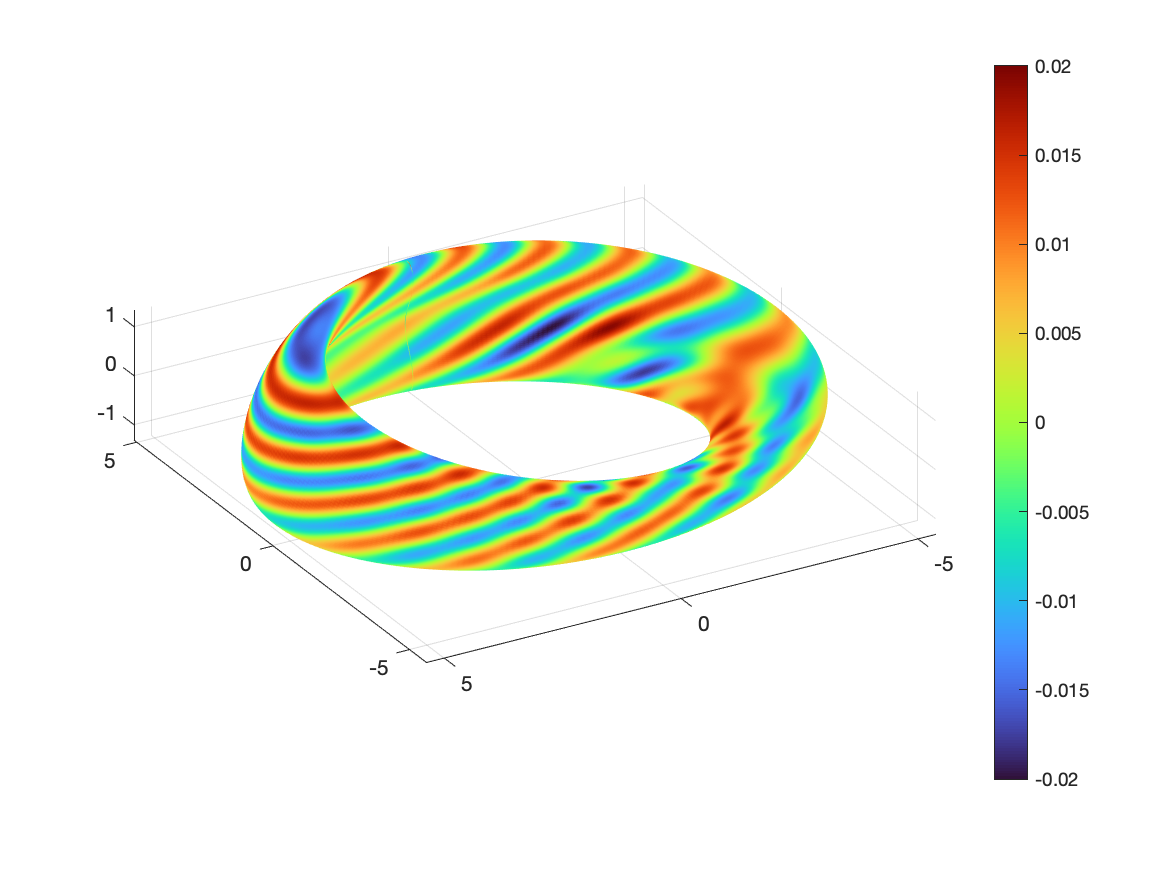}
    \includegraphics[width=0.32\textwidth]{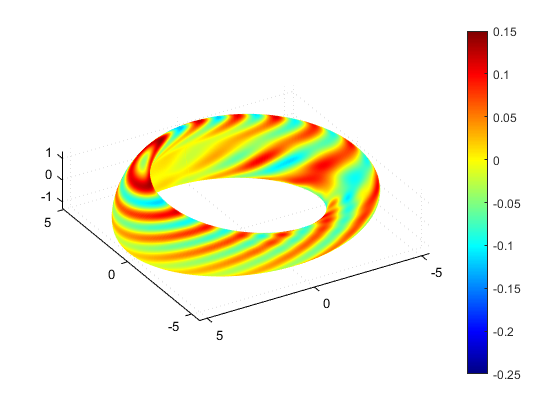}
    \caption{Example 4: Real parts of $\phi$ and $\psi$ on $\Gamma$ for $k=8$ for ${\cal S}$ (left), ${\cal K}$ (middle) and ${\cal K}'$ (right); $N_p$ and $2M$ are chosen according to the last row in Table~\ref{tab:ex4}.}
    \label{picfield4}
\end{figure}
plots the real parts of the unknown density functions $\phi$ in equations~\eqref{eq:bie:S} and~\eqref{eq:bie:Kpie} and $\psi$ in equation~\eqref{eq:bie:Kie} on $\Gamma$.

\section{Conclusion}\label{sec6}
In this paper, we proposed a high-order, FFT-accelerated
BIE method for 3D acoustic scattering by general smooth surfaces. Utilizing the periodicity of the unknown density in the azimuthal direction, we Fourier transformed the governing BIE in the azimuthal variable and obtained coupled integral equations with curve integral operators, where the kernels are exactly Fourier coefficients of 3D fundamental kernels.
We developed fast and accurate algorithms with complexity at most ${\cal O}(M\log M)$ for evaluating the ${\cal O}(M)$ Fourier coefficients of the singular kernels in the integral operators.  Numerical experiments have demonstrated the effectiveness and spectral accuracy of the new approach.

There are at least two research directions. A natural extension would be developing fast algorithms for the discretization of BIEs for 3D electromagnetic and elastic wave scattering by smooth surfaces. Besides, in order to develop fast solvers for the final linear system, we wish to investigate if there exist underlying low-rank structures for the curve integral operators corresponding to the interaction of two different Fourier indices, which in fact vanishes for axisymmetric surfaces. We shall report these works in the future.

\bibliographystyle{siamplain}
\bibliography{helm.bib}

\end{document}